\documentclass[12pt,leqno]{amsart}
\usepackage[left=3cm,right=3cm,top=3cm,bottom=3cm]{geometry}
    
\usepackage{latexsym}
\usepackage{amssymb}
\usepackage{amsmath}
\usepackage{amsfonts}
\usepackage{mathtools} 
\usepackage{changepage} 
\usepackage{xcolor}

\def\R{\mathbb{R}}

\newcommand{\bmat}{\left[\begin{matrix}}
\newcommand{\emat}{\end{matrix}\right]}
\usepackage[latin1]{inputenc} 
\usepackage{amsthm}
\usepackage{amsxtra}
\usepackage{amscd}
\usepackage{setspace}   
\usepackage{mathrsfs}   

\theoremstyle{remark}
\newtheorem{remark}[subsection]{Remark}

\theoremstyle{definition}

\newcommand{\Z}{\mathbb{Z}}
\newcommand{\Q}{\mathbb{Q}}
\newcommand{\C}{\mathbb{C}}

\newcommand{\F}{\mathbb{F}}

\newcommand{\sfr}[2]{\tfrac{#1}{#2}} 

\newcommand{\Mat}{\mathrm{Mat}}
\newcommand{\Sp}{\mathrm{Sp}}
\newcommand{\GL}{\mathrm{GL}}
\newcommand{\SL}{\mathrm{SL}}

\newcommand{\SO}{{\mathrm{SO}}}

\newcommand{\diag}{\mathrm{diag}}

\newcommand{\im}{\mathrm{im}\,}
\providecommand{\ker}{}
\renewcommand{\ker}{\mathrm{ker}\,}

\newcommand{\rk}{\mathrm{rk}}

\newcommand{\spn}{\mathrm{span}}

\newcommand{\pr}{\mathrm{pr}}

\newcommand{\ignore}[1]{{}}

\newcommand{\radon}{\mathcal{R}}
\newcommand{\xlat}{\mathsf{X}}
\newcommand{\ylat}{\mathsf{Y}}

\usepackage{enumerate}
\usepackage{hyperref}
\usepackage{pst-node}
\usepackage{tikz-cd}

    \title[Second moments of symplectic Siegel transforms]{Explicit Second Moments of \\ Symplectic Siegel Transforms}
    
    \author{Kristian Holm, Seungki Kim}

\makeatletter
\def\@seccntformat#1{%
  \protect\textup{\protect\@secnumfont
    \ifnum\pdfstrcmp{subsection}{#1}=0 \bfseries\fi%
    \csname the#1\endcsname
    \protect\@secnumpunct
  }%
}
\makeatother

\swapnumbers

\makeatletter
\def\swappedhead#1#2#3{%
  \thmnumber{\@upn{\the\thm@headfont#2\@ifnotempty{#1}{.~}}}%
  \thmname{#1}%
  \thmnote{ {\the\thm@notefont(#3)}}%
}
\makeatother

\newtheorem{theorem}[subsection]{Theorem}
\newtheorem{proposition}[subsection]{Proposition}
\newtheorem{lemma}[subsection]{Lemma}

\begin{document}

\maketitle

\begin{abstract}
We give two alternate formulations for the second moment formula of the Siegel transform over $\Sp(2n, \Z) \backslash \Sp(2n,\R)$, previously established by Athreya-Konstantoulas and Kelmer-Yu. We also provide several applications that demonstrate their usefulness.
\end{abstract}

\section{Introduction}

\subsection{Motivation}

This note represents an effort to better understand the Rogers integral formula for the second moment of the Siegel transform over symplectic groups, previously proved in various guises by Athreya-Konstantoulas \cite{AK16} as well as Kelmer-Yu \cite[Theorem 1]{KY17}. Elegant as these results are, practitioners have sometimes encountered difficulties in utilizing them.

First, \cite{AK16} seems to contain a number of inaccuracies, in its main statement \cite[Theorem 1.2]{AK16} as well as its proof; and although we believe these issues can be fixed, this would take a little more than a trivial amount of effort. Second, in the case of the symplectic second-moment formula \cite[Theorem 1]{KY17} for the primitive Siegel transform, the presence of a ``mysterious'' $L^2$-isometry $\iota$ \cite[Eq. (5.11)]{KY17}, constructed from a period operator $\mathcal{P}_f$ that averages the Siegel transform of $f$ over a quotient of a horospherical subgroup, has made the formula difficult to apply in certain situations (see, e.g., \cite{Hol23} and \cite[Remark 5.4]{BG23}). For comparison, the corresponding isometry in Rogers' primitive second-moment formula \cite[Theorem 5]{Rog55} for the space $\xlat_n = \Gamma \backslash G = \SL(n, \Z) \backslash \SL(n, \R)$ is essentially the operator $f \mapsto f \circ (-1)$; in particular, for even test functions (the class considered in \cite{KY17}), this isometry is trivial.

In addition to the matter of applying these formulas, the task of extending the methods used to prove them, to obtain formulas for the third and higher moments, seems rather daunting --- classifications of $\Sp(2n,\Z)$-orbits and their individual weight computations (for \cite{AK16}) and spectral theory (for \cite{KY17}). Hence, we believe that a case can be made for pursuing an alternative route to the symplectic second moment of the Siegel transform.

To understand the extra complications in the $\Sp(2n, \R)$-case, which are absent in the $\SL(n, \R)$-case, let us briefly recall the popular approach to the Siegel transforms due to Siegel and Weil \cite{Sie45,Wei46}. For $G$ a Lie group, $\Gamma \subseteq G$ a lattice, a closed subgroup $P \subseteq G$, and a ``nice'' function $f$ on $P\backslash G$, one can show with the unfolding technique that
\begin{align*}
\int_{\Gamma\backslash G} \sum_{\gamma \in \Gamma \cap P \backslash \Gamma} f(\gamma g) dg = \int_{P\backslash G} f(h)dh \cdot \int_{\Gamma \cap P \backslash P} d\gamma,
\end{align*}
where $dg,dh,d\gamma$ are the appropriate invariant measures on their respective spaces.

For instance, suppose $G = \SL(n,\R)$ and $P$ is the stabilizer of 
$\left( \begin{smallmatrix}
    0 & \cdots & 0 & 1 & 0 \\ 0 & \cdots & 0 & 0 & 1
\end{smallmatrix} \right)$
under the action of $G$ on $\mathrm{Mat}_{2\times n}(\R)$ by right multiplication. It turns out this action is transitive up to a nullset, which allows one to identify $P\backslash G \cong (\R^n)^2$ up to a nullset. Thus, under this identification, $dh$ is a multiple of the Lebesgue measure; a suitable choice of the test function $f$ would help determine the constant factor. This (almost) immediately gives the following well-known formula due to Rogers \cite{Rog55} concerning the space $\xlat_n$.

\begin{theorem}[\cite{Rog55}, Theorem 4; $k = 2$] \label{thm:RogersIntegrationFormula} 
If $f_1, f_2 : \R^n \rightarrow \R$ are non-negative functions with compact support, and $n \geq 3$, then
\begin{align*}
    \int_{\xlat_n} 
    \sum_{\substack{z_1, z_2 \in \Z^n \\ z_1, z_2 \neq 0}}
    f_1 (z_1 g) f_2 (z_2 g) \, d g 
    &= 
    \widehat{f}_1 (0) \widehat{f}_2 (0) 
    + \frac{1}{\zeta(n)}\sum_{\substack{j\geq 1 \\ k \neq 0}} \int_{\R^{n}}f_1(jz)f_2(kz) \, dz.
\end{align*} 
\end{theorem}
(Here and in the remainder of the paper, we write $\smash{\widehat{f}}(\xi) = \int_{\R^n} f(x)\exp(-2 \pi i x \cdot \xi) \, d x$ for the Fourier transform of a function $f$.) However, it must be pointed out that one is rather ``lucky'' in the $\SL(n, \R)$-case: In general, there is no reason to expect that the measure $dh$ on $P\backslash G$ is easy to determine concretely, nor that a single $G$-orbit fills the entire space being acted on. Indeed, when $G=\Sp(2n,\R)$, its right action on $\mathrm{Mat}_{2\times2n}(\R)$ partitions the space into infinitely many orbits of the form
\begin{align*}
\underbracket[1pt]{\begin{pmatrix}
0 & \cdots & 0 & 0 & 0 & \cdots & 0 & 1 \\
0 & \cdots & 0 & c & 0 & \cdots & 1 & 0
\end{pmatrix}}_{c \text{ appearing in the $(2,n)$'th entry}}
\cdot \, G = \left\{(x,y) \in \R^{2n}\times\R^{2n}: \sum_{i=1}^n(x_iy_{n+i}-x_{n+i}y_i) = c\right\}
\end{align*}
for each $c\in\R$ (up to orbits of smaller dimension). Thus, it does not seem easy to use this orbit classification as a starting point, and then -- for each of these orbits -- find a practically useful expression for the invariant measure and compute its weight, to arrive at a symplectic analogue of Theorem \ref{thm:RogersIntegrationFormula}.\footnote{It may be said that the results of \cite{AK16} are obtained via this path, by decomposing the sum $\smash{\sum_{\Gamma\cap P\backslash\Gamma}}$ into finer pieces. However, as already indicated, certain issues in the argument must be clarified: in the case of the proof of \cite[Lemma 5.1]{AK16} concerning the invariant measure, $y^*,y_i$ are apparently not well-defined.}

\subsection{This work}
We take a different approach to integrating products of Siegel transforms, which first appeared in Rogers' works \cite{Rog47,Rog55}, and was recently further developed and utilized by the second-named author \cite{Kim23,Kim24-2}, as well as \cite{GSVV25-2}. In a nutshell, this method goes as follows: Choose a certain family of Hecke operators $T_p$ indexed by prime $p$.
One can show, for a function $F$ on $\Gamma\backslash G$ that satisfies some mild conditions, that
\begin{align*}
T_pF \longrightarrow \frac{1}{\mathrm{vol}(\Gamma\backslash G)}\int_{\Gamma\backslash G} F(g)dg
\end{align*}
as $p \rightarrow \infty$.\footnote{This is reminiscent of the equidistribution of the Hecke points --- see e.g., \cite{COU01}. It cannot be used directly in our situation, however, since our $F$ diverges at the cusp.} Thus one would like to evaluate the $p$-limit of $T_pF$ for $F$ equal to the second-moment Siegel transform
\begin{align*}
F(g) = \sum_{z_1,z_2\in\Z^{2n}\setminus \{0\}} f(z_1g,z_2g)
\end{align*}
of a measurable function $f : (\R^{2n})^2\rightarrow\R$. 
It turns out that this objective reduces to the problem of counting lattice points inside a bounded set, subject to certain constraints. 

For $G = \Sp(2n,\R)$ and $\Gamma = \Sp(2n,\Z)$ and hence $\ylat_{2n} :=\Gamma \backslash G = \Sp(2n,\Z)\backslash\Sp(2n,\R)$, we solve this counting problem in two different ways, yielding two different expressions for the second moment formula.

\begin{theorem}\label{thm:main}
Let $n \geq 2$. Let $\ylat_{2n}$ be equipped with the right $\Sp(2n,\R)$-invariant probability measure $\mu$, and let $f:(\R^{2n})^2\rightarrow\R$ be a Borel measurable function. Then
\begin{align*}
\int_{\ylat_{2n}}\sum_{z_1,z_2\in\Z^{2n}\setminus \{0\}} f(z_1g,z_2g) \, d\mu 
= 
\frac{1}{\zeta(2n)} \sum_{\substack{j \geq 1 \\ k \in \Z}} 
\int_{\R^{2n}}\int_{L(k,z_1)}f(jz_1,z_2) \, d z_2 \, \frac{dz_1}{\|z_1\|} \\
+ \, \frac{1}{\zeta(2n)}\sum_{\substack{j \geq 1 \\ k \neq 0}}\int_{\R^{2n}}f(jz,kz) \, dz,
\end{align*}
where $J=\left( \begin{smallmatrix} & I_n \\ -I_n & \end{smallmatrix}\right)$ is the matrix for the symplectic form used to define $\Sp(2n, \R)$, and 
\begin{align*}
L(k,z_1) 
= 
\spn_\R(z_1J)^\perp +\frac{kz_1J}{\|z_1\|^2}
= 
\big\{ z_2 \in \R^{2n} : z_2 \cdot z_1 J = k \big\}.
\end{align*}
\end{theorem}
Our next expression for the second moment strongly resembles the right-hand side of Theorem \ref{thm:RogersIntegrationFormula}, but contains an additional term; we can think of this as a distinctly symplectic signature that quantifies the added complexity (compared to $\xlat_{2n}$) of averaging over the much smaller space $\ylat_{2n}$. (We return to this matter in Section \ref{subsect:application_discrepancy}.)
\begin{theorem}\label{thm:main-Schwartz}
Consider the special case where $f(z_1,z_2) = f_1(z_1)f_2(z_2)$ is the product of two functions $f_1, f_2 : \R^{2n} \rightarrow \C$ which are either
\begin{enumerate}
    \item in the Schwartz class, 
    \item compactly supported and Riemann integrable,
    \item or Borel measurable, compactly supported, and bounded.
\end{enumerate}
Then with the same notation as above, we have
\begin{align*}
&\int_{\ylat_{2n}}\sum_{z_1,z_2\in\Z^{2n}\setminus \{0\}} f_1(z_1g) f_2(z_2g) \, d\mu (g) \\
&\quad = \widehat{f}_1 (0) \widehat{f}_2 (0) 
+ \frac{1}{\zeta(2n)}
\sum_{\substack{j \geq 1 \\ k \neq 0}}
\int_{\R^{2n}}f_1(jz) f_2(kz) \, dz
+ \frac{1}{\zeta(2n)}
\sum_{\substack{j \geq 1 \\ k \neq 0}}
\int_{\R^{2n}}f_1(jz)\widehat{f}_2(kzJ) \, dz.
\end{align*}
\end{theorem}
\begin{remark}
For $f(z_1,z_2)$ of the form $f(z_1)f(z_2)$, the innermost integral in Theorem \ref{thm:main} over $L(k,z_1)$ can be seen as the Radon transform. This gives a natural connection between Theorems \ref{thm:main} and \ref{thm:main-Schwartz} via the Fourier slice theorem, which relates the Radon and Fourier transforms. We elaborate on this in Section \ref{subsect:relation_between_formulas}
\end{remark}
\begin{remark}
    Previous work by the first-named author used the results of \cite{KY17} to produce a prototype \cite[Thm. 5.3 and Prop. 5.5]{Hol23} of Theorem \ref{thm:main}, where the geometry of the right-hand side of the second-moment formula and the role of the Radon transform were less explicit than in Theorem \ref{thm:main}. However, the methods used in the current paper differ significantly from those of \cite{KY17} and \cite[Sect. 5]{Hol23}.
\end{remark}
\begin{remark}\label{remark:standard_vs_symplectic_Rogers} 
    Even though the new symplectic term may look asymmetric in $f_1$ and $f_2$, this is not the case: This follows from Parseval's theorem, the evenness of our test functions, and the fact that $J^{-1} = -J$. Therefore, the entire right-hand side of Theorem \ref{thm:main-Schwartz} is symmetric in $f_1$ and $f_2$ (as it should be).
\end{remark}

As is well-known, Theorem \ref{thm:RogersIntegrationFormula} of Rogers has higher-rank analogues, also proved in \cite{Rog55}. Since the first $n$ moments over $\ylat_{2n}$ are known to exist (cf. \cite[Theorem 1.2]{Hol23}), it is natural to ask about explicit formulas such as Theorems \ref{thm:main} and \ref{thm:main-Schwartz} for some or all of these higher moments. We hope to return to this problem in the future with the machinery developed in the present paper.

\subsection{Applications}

We provide four applications of our main result to demonstrate their usefulness in studying symplectic random lattices. The first three start from Theorem \ref{thm:main} and come down to estimates involving Bessel functions. The last one uses Theorem \ref{thm:main-Schwartz}. 

First, we upgrade Kelmer--Yu's bound on the mean square discrepancy of symplectic lattice-point counting \cite[Theorem 1.2 (0.11)]{KY17} to an asymptotic formula, in case $B$ is a ball of radius $R$ at origin (cf. also \cite[Theorem 1.3]{AK16}). In the notations of \cite{KY17} (also see Section \ref{subsect:application_discrepancy} below), we prove
\begin{align*}
\int_{\ylat_{2n}} |D(\Z^{2n}g,B)|^2d\mu_n = \frac{1}{\mathrm{vol}(B)}\left(\frac{4\zeta(2n-1)}{\zeta(2n)}-2+O_n(R^{-3})\right).
\end{align*}
The main term here is identical to its $\xlat_{2n}$-counterpart. Moreover, for $n\geq 4$, we can show that the error oscillates like $R^{-3}\sin(2\pi R^2)$ --- see Theorem \ref{thm:mean_sq_disc} below for the precise statements.

We also deal with a technical difficulty that was left unresolved in \cite{Hol23} and prove that the variance $\sigma^2$ from \cite[Theorem 1.1]{Hol23} is $2+o_n(1)$ --- see Theorem \ref{thm:holm_improvement} below. Previously, it was unclear even whether $\sigma \neq 0$. On a related note, Holm's construction of the set $\Omega_T$ can be given a natural Diophantine interpretation, and thus one may expect an symplectic analogue of the central limit theorem of Bj\"orklund--Gorodnik \cite{BG19} where sets of a similar Diophantine flavor are studied --- we will return to this topic in a forthcoming paper. 

We give two more applications that may be of interest in (but possibly also outside of) lattice-based cryptography: We prove that for all sufficiently large $n$, the shortest non-zero vector length in almost all lattices in $\ylat_{2n}$ is $\approx \smash{\sqrt{\pi e / n}}$ --- see Theorem \ref{thm:symplectic_GH} for a precise statement. This phenomenon is a rigorous instantiation of what is usually referred to as the \emph{Gaussian heuristic}, proved for $\SL(n,\R)$ in \cite[Theorem 6]{LN25} and for $\SL(n,K\otimes\R)$ for a number field $K$ in \cite[Theorem 3]{GSVV25}.

The last application concerns Gaussian theta functions and the \emph{smoothing parameter} $\eta_\varepsilon(L)$ of a lattice $L \subseteq \R^n$, first introduced in Micciancio--Regev \cite{MR07}; see also \cite{RS24}. Roughly speaking, this quantity measures the threshold at which the function
$$\varrho_s(L)=\sum_{v \in L} \exp(-\pi \|v\|^2/s^2)$$
satisfies $(\det L) \varrho_s (L) \approx s^n.$
We prove a second moment formula for $\varrho_s(L)$, and thereby obtain that $\eta_\varepsilon(L) = 1 + O(1/n)$ for almost all symplectic lattices $L \in \ylat_{2n}$ --- see Theorem \ref{thm:lattice_smoothing_parameter}. This is essentially an extension of \cite[Corollary 3]{PS26} to symplectic lattices. In doing this, we use Theorem \ref{thm:main-Schwartz} to compute the $L^2$-norm of $\varrho_s(L)$ over $\ylat_{2n}$, which may be of broader number-theoretic interest.

\subsection{Acknowledgment}

We thank Anish Ghosh and Jiyoung Han for their contributions at an early stage of this project. We would also like to thank Paul Nelson, Anders S\"odergren, and Shucheng Yu for helpful conversations and suggestions. The second-named author was partially supported by NSF CNS-2034176.

\section{Basics}

\subsection{The symplectic group}

Throughout let $n \geq 2$. We define $I_n$ to be the $n \times n$ identity matrix, and let
\begin{align*}
J = J_{2n} = \begin{pmatrix} 0 & I_n \\ -I_n & 0 \end{pmatrix}.
\end{align*}
The matrix $J$ defines an alternating bilinear form on $\R^{2n}$, given by
\begin{align*}
    \omega(z_1, z_2) = z_1 J z_2^\intercal = x_1 \cdot y_2 - x_2 \cdot y_1 \qquad (\text{for } z_i = (x_i, y_i) \in \R^n \times \R^n ).
\end{align*}
For any commutative ring $R$ with $1 \neq 0$, we define the group
\begin{align*}
\Sp(2n,R) = \{M \in \Mat_{2n \times 2_n}(R) : MJM^\intercal = J\},
\end{align*}
which consists of all $R$-linear transformations that preserve $\omega$. In the case $R = \R$, we obtain the \textit{symplectic group}, which is a semisimple Lie group of dimension $n(2n+1).$

If $L \subseteq \R^{2n}$ is a lattice (of full rank), we say that $L$ is \textit{symplectic} if $\omega(L, L) \subseteq \Z$. By using a symplectic version of the Gram--Schmidt process, one can show (see \cite[Prop. 2.4]{Hol23}) that 
\begin{align*}
    L \text{ is symplectic} \Longleftrightarrow L = \Z^{2n} g \text{ for some } g \in \Sp (2n, \R).
\end{align*}
The space $\ylat_{2n}$ of all symplectic lattices in $\R^{2n}$ can therefore be identified with the quotient space $\Sp(2n, \Z) \backslash \Sp(2n, \R).$ 

In what follows, let us write $G = \Sp(2n,\R), \Gamma = \Sp(2n,\Z),$ and $\ylat_{2n} = \Gamma \backslash G$. Since semisimple groups and discrete groups are unimodular (as locally compact groups), the quotient $\Gamma \backslash G = \ylat_{2n}$ inherits a right-$G$-invariant measure from the Haar measure on $G$. In fact, by reduction theory on $G$ (see, e.g., \cite[Appendix A]{Hol23}), this measure is finite and can therefore be normalized to yield a probability measure $\mu$ on $\ylat_{2n}.$ We thus endow $\ylat_{2n}$ with this right $G$-invariant probability measure $\mu$. 

It is clear that any $M \in \Sp(2n, \R)$ must be invertible (with determinant $\pm 1$). As a matter of fact, we have $\Sp(2n, \R) \subseteq \SL(2n, \R)$, which induces an inclusion of quotient spaces, and identifies $\ylat_{2n}$ with a submanifold of $\xlat_{2n}$ of strictly smaller dimension. In particular, $\ylat_{2n}$ is a nullset in the space $\xlat_{2n}.$ 

\subsection{Hecke operators on $\ylat_{2n}$}

Let $p$ be a prime and write $\Lambda = \SL(n,\Z)$. For $i=0,\ldots, n$, let $\Sigma_{i,p}$ be the set of all $n \times n$ upper-triangular integral matrices $(a_{jk})_{1\leq j,k \leq n}$ for which there exists $1 \leq k_1 < \ldots <k_i \leq n$ such that 
\begin{enumerate}[(i)]
\item $a_{kk} = p$ if $k =k_\ell$ for some $\ell$, otherwise $a_{kk} = 1$;

\item for $j < k_\ell$, $0 \leq a_{j k_\ell} < p$, unless if $j = k_{\ell'}$ for $\ell' \neq \ell$, in which case $a_{hk_\ell}=0$;

\item all the remaining entries are zero.
\end{enumerate}
As is well-known, it can be shown by linear algebra (see for example \cite[Lemma 3.3]{Kim24}) that $\Sigma_{i,p}$ forms the set of the right coset representatives of $\Lambda$ in the double coset
\begin{align*}
\Lambda \cdot \diag(\underbracket[1pt]{1,\ldots,1}_{n-i},\underbracket[1pt]{\, p,\ldots,p}_i) \cdot \Lambda.
\end{align*}

It is shown in \cite[Lemma 3.49]{And09} that each right coset representatives of the double coset
\begin{align}\label{eq:ourcoset}
\Gamma \begin{pmatrix} I_n & \\  & pI_n\end{pmatrix}\Gamma
\end{align}
can be written in the form
\begin{align*}
\begin{pmatrix} pD^{-\intercal} & B \\ & D \end{pmatrix},
\end{align*}
where $B,D \in \mathrm{Mat}_{n\times n}(\Z)$ with $D \in \Sigma_{i,p}$ for some $i=0,\ldots,n$, and $B$ satisfies $B^\intercal D = D^\intercal B$ and can be chosen uniquely modulo the equivalence relation
\begin{align} \label{eq:eqclass_B}
B \sim B' \mbox{ (mod $D$)} \Leftrightarrow (B-B')D^{-1} \in \mathrm{Mat}_{n\times n}(\Z).
\end{align}

For each $i$, let us refer to the set of such representatives as $\Delta_{i,p}$, and write $\Delta_p = \bigcup_i \Delta_{i,p}$. Clearly

\begin{align*}
\Delta_{0,p} = \left\{ \begin{pmatrix} p I_n & \\ & I_n\end{pmatrix} \right\}, 
\quad \Delta_{n,p} = \left\{\begin{pmatrix} I_n & B \\  & pI_n\end{pmatrix} : 
\! \begin{array}{l}
     B=(b_{jk}), \, b_{jk}=b_{kj}, \\
     \text{and } 0\leq b_{jk} < p  
\end{array} \right\}.
\end{align*}

It remains to describe $\Delta_{i,p}$ for the intermediate $i$'s. Given $D \in \Sigma_{i,p}$, there exists a permutation matrix (that is, integral and orthonormal) $O$ of order $2$ such that
\begin{align*}
ODO^\intercal = \begin{pmatrix} I_{n-i} & \tilde D \\ & pI_i \end{pmatrix},
\end{align*}
and the conditions on $B$ --- $B^\intercal D = D^\intercal B$ and uniqueness under \eqref{eq:eqclass_B} --- are preserved under conjugation by $O$. Writing $\{1,\ldots,n\} = \{k_1<\ldots <k_i\} \coprod \{k'_1 <\ldots <k'_{n-i}\}$ where the $k_i$'s are determined by $D \in \Sigma_{i,p}$ as described earlier, such $O$ can be chosen to correspond to the permutation $(k_1 k'_{n-i})\ldots(k_rk'_{n-r})$ for the largest $r$ such that $k_r \leq n-i$. Hence it suffices to consider the special case in which
\begin{align} \label{eq:D_std}
D = \begin{pmatrix} I_{n-i} & \tilde D \\ & pI_i \end{pmatrix}
\end{align}
and conjugate back by $O^{-1}=O^\intercal =O$ to obtain the general case.

\begin{lemma} \label{lemma:reps}
Let $D$ be as in \eqref{eq:D_std}. Then one can choose the set of representatives $B$ of the equivalence classes of the relation \eqref{eq:eqclass_B} satisfying $B^\intercal D = D^\intercal B$ to consist of the matrices of the form
\begin{align*}
\begin{pmatrix} 0 & 0 \\ 0 & B_{22}\end{pmatrix},
\end{align*}
where $B_{22}$ is an $i \times i$ integral symmetric matrix with entries in $\{0,\ldots,p-1\}$.
\end{lemma}

As a consequence, in general, an element of $\Delta_{i,p}$ has the form
\begin{align*}
\begin{pmatrix} O & \\ & O \end{pmatrix}\begin{pmatrix} pD^{-\intercal} & B \\ & D \end{pmatrix}\begin{pmatrix} O & \\ & O \end{pmatrix}
\end{align*}
where $D$ is of the form \eqref{eq:D_std} and $B$ is as described in Lemma \ref{lemma:reps}.
\begin{proof}
Write
\begin{align*}
B = \begin{pmatrix} B_{11} & B_{12} \\ B_{21} & B_{22} \end{pmatrix},
\end{align*}
where $B_{22}$ is an $i \times i$ matrix, and the rest are sized accordingly. Since
\begin{align*}
B^\intercal D = 
\begin{pmatrix}
B_{11}^\intercal  & B_{11}^\intercal \tilde D+pB_{21}^\intercal  \\ B_{12}^\intercal  & B_{12}^\intercal \tilde D + pB_{22}^\intercal 
\end{pmatrix},
D^\intercal B =
\begin{pmatrix}
B_{11} & B_{12} \\ \tilde D^\intercal B_{11} + pB_{21} & \tilde D^\intercal B_{12} + pB_{22}
\end{pmatrix},
\end{align*}
and these two matrices must be equal, $B_{11}$ and $B_{12}^\intercal \tilde D + pB_{22}^\intercal $ must be symmetric, and $B_{12} = B_{11}^\intercal \tilde D +pB_{21}^\intercal $.

One also computes that
\begin{align*}
BD^{-1} =
\begin{pmatrix}
B_{11} & -p^{-1}B_{11}\tilde D + p^{-1}B_{12} \\ B_{21} & -p^{-1}B_{21}\tilde D + p^{-1}B_{22}
\end{pmatrix},
\end{align*}
and thus $B \sim 0$ means that $B_{11}, B_{21}$ are integral and $B_{12} \equiv_p B_{11}\tilde D, B_{22} \equiv_p B_{21}\tilde D$. From these facts, one can check that for any symmetric $i \times i$ matrix $A$, $B$ is equivalent to
\begin{align*}
B -
\begin{pmatrix}
B_{11} & B_{11}\tilde D \\ 0 & 0
\end{pmatrix} -
\begin{pmatrix}
0 & pB_{21}^\intercal  \\ B_{21} & B_{21}\tilde D + pA
\end{pmatrix}
=\begin{pmatrix}
0 & 0 \\ 0 & B_{22}'
\end{pmatrix},
\end{align*}
and that $B_{22}'$ is symmetric. By adjusting $A$, $B_{22}'$ can be chosen uniquely (mod $p$).
\end{proof}

It is clear that $|\Delta_{n,p}| = p^{n(n+1)/2}$, and that in general $|\Delta_{i,p}| \ll_n p^{(n-i)i+i(i+1)/2}$, whose exponent is increasing in the range $0 \leq i \leq n$. Therefore, $|\Delta_p| =p^{n(n+1)/2}(1+O_n(p^{-1}))$, and for all sufficiently large $p$, $\Delta_{n,p}$ has the dominating contribution.

Let us denote by $T_p$ the Hecke operator on $\ylat_{2n}$ associated with the double coset \eqref{eq:ourcoset}, namely,
\begin{align*}
T_pF(g) = \frac{1}{|\Delta_p|}\sum_{h \in \Delta_p}F(p^{-\frac{1}{2}}h g)
\end{align*}
for a function $F$ on $\ylat_{2n}$. The following lemma is crucial for our approach.

\begin{lemma}\label{lemma:eqdist}
Suppose $F$ is a measurable function on $\ylat_{2n}$, and that $\lim_{p \rightarrow \infty} T_pF(g)$ exists and has the same value $I <\infty$ for almost every $g \in \ylat_{2n}$. Then $F$ is integrable, and $I = \int_
{\ylat_{2n}} Fd\mu$.
\end{lemma}
\begin{proof}
The proof is identical to the proofs of Lemmas 6 and 7 of Kim \cite{Kim23}, after only the following small adjustment: in the proof of \cite[Lemma 6]{Kim23}, take
\begin{align*}
&\alpha_p = \begin{pmatrix} I_n & \\ & pI_n \end{pmatrix}, \quad \mathcal M_p = \Delta_p \mbox{ as we defined above}, \quad \widetilde{\mathcal M}_p = \{h^\intercal :h \in \Delta_p\}.
\end{align*}

Also, notice that, in the proof of \cite[Lemma 7]{Kim23}, it is unnecessary to assume that the function $\rho$ there ($F$ in our notation) is integrable --- it follows from other assumptions, as the proof there shows --- and that the condition $\rho \geq 0$ can be dropped, since any measurable function is the difference of two nonnegative measurable functions.
\end{proof}

\subsection{Outline}

Let $f:(\R^{2n})^2 \rightarrow \R$ be a bounded and compactly supported Riemann integrable function. The bulk of this paper is devoted to the study of $T_pF(g)$ for
\begin{align} \label{eq:2strans}
F(g) = \sum_{z_1,z_2 \in \Z^{2n}\setminus \{0\}} f(z_1g,z_2g).
\end{align}
In particular, we will show that it converges to the same value in the $p$ limit, regardless of $g$. This allows us to apply Lemma \ref{lemma:eqdist} and thereby obtain the symplectic Rogers formula for $f$. Standard approximation arguments then extend the formula to any measurable $f$.

It will be shown in Section \ref{sec:Delta_i} below that as $p \rightarrow \infty$
\begin{align*}
\frac{1}{|\Delta_p|}\sum_{h \in \Delta_{i,p}} F(p^{-\frac{1}{2}}hg)
\end{align*}
vanishes unless $i=n$. Hence it suffices to consider
\begin{align*}
\frac{1}{p^{n(n+1)/2}}\sum_{h \in \Delta_{n,p}}F(p^{-\frac{1}{2}}hg),
\end{align*}
or equivalently, unpacking $\Delta_{n,p}$ and $F$,
\begin{align}\label{eq:toest_main}
\frac{1}{p^{n(n+1)/2}}\sum_{B} \sum_{z_1,z_2 \in \Z^{2n}\setminus \{0\}} f(p^{-\frac{1}{2}}(x_1;x_1B+py_1)g,p^{-\frac{1}{2}}(x_2;x_2B+py_2)g),
\end{align}
where we write $z_i=(x_i;y_i)$ for $x_i,y_i \in \Z^n$, and the sum over $B$ runs over all $n \times n$ symmetric matrices with entries in $\{0,\ldots,p-1\}$.

We further parse the above sum according to whether the $\F_p$-rank of $\{x_1,x_2\}$ is $1$ or $2$. The rank $1$ case will be handled in Section \ref{sec:mainest_rk1}. The rank $2$ case will be evaluated in two different ways, in Sections 4.2-3 and 4.4, respectively. It turns out that the two distinct formulas thus obtained are succinctly related by the Fourier slice formula.

The following statement, of which we will only use the cases $k=1,2$, forms another important component of our machinery.

\begin{proposition}\label{prop:eqdist}
Let $x = \{x_1,\ldots,x_k\} \subseteq \F_p^n$ be linearly independent. Consider the map $T_x: \smash{\F_p^{n(n+1)/2}} \rightarrow \F_p^{kn}$ defined by $B \mapsto (x_1B,\ldots,x_kB)$.
Then
\[
\im T_x=\left\{(y_1, \ldots, y_k)\in \F_p^{kn}: (x_i;y_i)J(x_j;y_j)^\intercal  \equiv 0 \, \, (\mathrm{mod} \, p) \text{ for all } 1 \leq i,j \leq k\right\}.
\]
In particular, $\dim \ker T_x = (n-k)(n-k+1)/2$.
\end{proposition}

\begin{proof}
The image of $T_x$ does satisfy $(x_i;y_i)J(x_j;y_j)^\intercal = (y_i,y_j) \cdot(x_j,-x_i) \equiv 0$ (mod $p$), so it remains to show that it surjects. By the change of coordinates, we may assume $x_i=e_i$, where $e_i \in \F_p^n$ is the vector whose $i$'th entry is $1$ and all the other entries are $0$.\\[-0.7em]
{\begin{minipage}[t]{0.7\textwidth}\vspace{-0pt}
More precisely, the diagram to the right commutes, where $e = \{e_1,\ldots,e_k\}$, $X$ is an element of $\GL(n,\F_p)$ whose first $k$ rows are $x_1,\ldots, x_k$, $c_X(B) = X^{-1}BX^{-\intercal}$, and \\[-1em]
\begin{align*}
    r_X(y_1,\ldots,y_k)= (y_1X^{-\intercal},\ldots,y_kX^{-\intercal}).\\[-1em]
\end{align*}
Hence it suffices to show that $T_e$ surjects.
\end{minipage} \hfill
\begin{minipage}[t]{0.25\textwidth}\vspace{0pt}
\begin{center} \vspace{0pt}
\begin{tikzcd}[column sep = large]
\F_p^{\frac{n(n+1)}{2}} \hspace{-3pt}\arrow{r}{T_e} \arrow[swap]{d}{c_X} & \F_p^{kn} \arrow{d}{r_X} \\
\F_p^{\frac{n(n+1)}{2}} \hspace{-3pt} \arrow{r}{T_x}& \F_p^{kn}
\end{tikzcd}
\end{center}
\end{minipage} \vspace{0.9pt}}
\par Suppose $(y_i,y_j)\cdot(e_j,-e_i)\equiv 0$ mod $p$ for all $1 \leq i,j \leq k$. In other words, $(y_i,e_j) \equiv (y_j,e_i)$, that is, the $i$'th entry of $y_j$ is equal to the $j$'th entry of $y_i$. This implies that $y_1,\ldots,y_k$ are the first $k$ rows of some symmetric $n \times n$ matrix $B$, and thus $T_e(B) = (y_1,\ldots,y_k)$. This completes the proof.
\end{proof}

\subsection{Polar, orthogonal, and quotient lattices} \label{subsect:polaretc}

Let $L \subseteq \R^n$ be a lattice, not necessarily of full rank. We equip $L$ with the restriction of the standard inner product on $\R^n$; for all lattices appearing in this paper, we assume this to be the case unless stated otherwise. We say $\Lambda \subseteq L$ is a \emph{primitive} sublattice if $\spn_\R(\Lambda) \cap L = \Lambda$. Following Schmidt \cite{Sch66}, we define the \emph{polar lattice} $L^P$ of $L$ as
\begin{align*}
L^P := \big\{ w \in \mathrm{span}_\R (L) : \langle v,w \rangle \in \Z \text{ for all } v \in L \big\}.
\end{align*}
This is more commonly known as the \emph{dual lattice}. It is useful to note that, for a full-rank lattice in $\R^n$ spanned by the row vectors of an $n\times n$ matrix $g$, a basis of its dual lattice is is given by the rows of $g^{-\intercal}$.

For a primitive sublattice $\Lambda \subseteq L$, the \emph{orthogonal lattice} $\Lambda^\perp \subseteq L^P$ is defined as
\begin{align*}
\Lambda^\perp = \big\{ w \in L^P : \langle v,w \rangle = 0 \text{ for all } v \in \Lambda \big\}.
\end{align*}
$\Lambda^\perp$ is also primitive. We have $\det \Lambda^\perp = \det \Lambda / \det L$ (see e.g. \cite[Section 2]{Thu93}).

We identify the \emph{quotient lattice} (or \emph{factor lattice}, as in \cite{HK23}) $L/\Lambda$ with the orthogonal projection of $L$ onto $\spn_\R(\Lambda)^\perp$, equipped with the restriction of the standard inner product on $\R^n$. Under this identification, for primitive $\Lambda$, $\spn_\R(\Lambda^\perp)=\spn_\R(L/\Lambda)$, and $\det L/\Lambda = \det L / \det \Lambda$. Thanks to the determinant formulas, it is easy to check that $(\Lambda^\perp)^P = L/\Lambda$.

\subsection{Miscellaneous notation}
We write $x \equiv_p y$ to mean $x \equiv y$ (mod $p$). $\mathrm{rk}_p$ ($\mathrm{rk}_\Z$) of a set means the $\F_p$-rank ($\Z$-rank) of the set. $\Z^n_\pr$ denotes the set of primitive vectors of $\Z^n$. For a function $f$ defined on a countable set $L$, we sometimes write $f(L) = \sum_{x \in L}f(x)$. As above, for $z \in \R^{2n}$, we use a semicolon to express $z$ as a concatenation of two vectors in $\R^n$ (so for example, $z=(x;y)$ signifies that $x,y\in\R^n$).

\section{Contribution from $\Delta_{n,p}$, $\F_p$-rank $1$}
\label{sec:mainest_rk1}

\subsection{A reduction}

In this section, we consider the expression
\begin{align*}
\frac{1}{p^{n(n+1)/2}}\sum_{B} 
\sum_{\substack{z_1,z_2 \in \Z^{2n} \setminus \{0\}\\ \mathrm{rk}_p\{x_1,x_2\}=1}} 
f(p^{-\frac{1}{2}}(x_1;x_1B+py_1)g,p^{-\frac{1}{2}}(x_2;x_2B+py_2)g),
\end{align*}
a part of the sum \eqref{eq:toest_main}.

Notice that in the above expression, for each $x_1$ there exists at most one $y_1$ such that $f$ does not vanish for $p \gg 0$, since $\mathrm{supp}\ f\circ g$ is has bounded diameter. 
By this fact and Proposition \ref{prop:eqdist} with $k=1$, the above sum can be written as
\begin{align} \label{eq:rk_1}
\frac{1}{p^n} \sum_{\substack{z \in \Z^{2n}\setminus \{0\} \\ x \neq 0}} \sum_{k=1}^{p-1} \sum_{y_2\in\Z^n} f(p^{-\frac{1}{2}}zg,p^{-\frac{1}{2}}(kz+py_2)g).
\end{align}

Consider for a moment the sum
\begin{align*}
\frac{1}{p^n} \sum_{\substack{z \in \Z^{2n}\setminus \{0\} \\ x = 0}} \sum_{k=1}^{p-1} \sum_{y_2\in\Z^n} f(p^{-\frac{1}{2}}zg,p^{-\frac{1}{2}}(kz+py_2)g).
\end{align*}
By a rough lattice-point counting estimate, the sum over $z=(0;y)$ contributes $O(p^{n/2})$. Moreover, the sum over $k$ contributes at most $O(p^{1/2})$, since $z$ is nonzero and every coordinate of $kz$ (mod $p$) must be at most $O(p^{1/2})$ for $f$ to not vanish. Therefore the whole expression is $O(p^{(-n+1)/2})$, which vanishes as $p \rightarrow \infty$ for any $n \geq 2$. The point is that instead of \eqref{eq:rk_1}, we can consider
\begin{align*}
\frac{1}{p^n} \sum_{z \in \Z^{2n}\setminus \{0\}} \sum_{k=1}^{p-1} \sum_{y_2\in\Z^n} f(p^{-\frac{1}{2}}zg,p^{-\frac{1}{2}}(kz+py_2)g).
\end{align*}
instead, or equivalently
\begin{align}\label{eq:rk1}
\frac{1}{p^n}\sum_{\substack{z_1,z_2\in\Z^{2n}\setminus \{0\} \\ \mathrm{rk}_p\{z_1,z_2\}=1}} f(p^{-\frac{1}{2}}z_1g,p^{-\frac{1}{2}}z_2g).
\end{align}

We divide this sum into two parts: (i) when $z_1$ and $z_2$ are multiples over $\Q$ (ii) when they are independent over $\Q$.

\subsection{When $\boldsymbol{z_1,z_2}$ are dependent over $\Q$}
\label{sec:rk1_dep}
Such a pair $(z_1,z_2)$ is identical to $(jz,kz)$ for primitive $z\in\Z^{2n}$, $j \geq 1, k \in\Z\setminus \{0\}$. Their contribution is therefore
\begin{align*}
\frac{1}{p^n}\sum_{\substack{j \geq 1 \\ k \neq 0}} \sum_{z \in \Z^{2n}_\pr} f(p^{-\frac{1}{2}}jzg,p^{-\frac{1}{2}}kzg).
\end{align*}
Provided this converges (absolutely) as $p \rightarrow \infty$, it is easy to see that it is a Riemann sum that converges to
\begin{align*}
\frac{1}{\zeta(2n)}\sum_{\substack{j \geq 1 \\ k \neq 0}} \int_{\R^{2n}} f(jz,kz)dz.
\end{align*}

To see that it converges, first consider the partial sum at $j=1$
\begin{align*}
\sum_{k \geq 1}\sum_{z \in \Z^{2n}_\pr} f(p^{-\frac{1}{2}}zg,p^{-\frac{1}{2}}kzg).
\end{align*}
This is bounded by a constant (depending only on $f$ and $g$) times
\begin{align*}
\sum_{z \in \Z^{2n}_\pr} 
\mathbf{1}_{B(p^{1/2})}(z)\frac{p^\frac{1}{2}}{\|z\|},
\end{align*}
where $B(p^{1/2}) \subseteq \R^{2n}$ is the ball at origin of radius $p^{1/2}$. By summation by parts (e.g., \cite[Lemma 4.1]{Kim24}), this can be seen to be $O(p^n)$. Similarly, the partial sum at other $j$ contributes $O(j^{-2n}p^n)$.

\subsection{When $\boldsymbol{z_1,z_2}$ are independent over $\Q$}
\label{sec:rk1_indep}
Let us recall the properties of the pair $(z_1,z_2)$ in this case: both $z_1,z_2$ has length $O(p^{1/2})$, and they are collinear mod $p$ but independent over $\Q$. The former implies that $\|z_1\wedge z_2\| = O(p)$, and the latter implies that $z_1\wedge z_2 = p(u_1 \wedge u_2)$ for linearly independent $u_1,u_2 \in \Z^{2n}$. Therefore, $\|u_1 \wedge u_2\| = O(1)$, and consequently, there exist $O(1)$ number of $u_1 \wedge u_2$ containing a pair $(z_1,z_2)$ with the properties of interest. We will show that, for any such plane $\mathcal P = u_1 \wedge u_2$ of covolume $O(1)$,
\begin{align*}
\frac{1}{p^n}\sum_{\substack{z_1,z_2 \in \mathcal P \\ \mathrm{rk}_p=1,\mathrm{rk}_\Z=2}}f(p^{-\frac{1}{2}}z_1g,p^{-\frac{1}{2}}z_2g)
\end{align*}
vanishes as $p\rightarrow\infty$. In other words, these $(z_1,z_2)$ do not contribute to the Rogers formula.

There exist $O(p)$ primitive vectors $z_1 \in \mathcal P$ of length $O(p^{1/2})$. To count the number of $z_2$ that potentially contribute for each primitive $z_1$, write $\mathcal P_\Z = \mathcal P\cap \Z^{2n}$, and consider the quotient map $\pi:\mathcal P_\Z\rightarrow\mathcal P_\Z/\langle z_1\rangle$ by orthogonal projection. Since $z_1\wedge z_2$ is a multiple of $p$, $\pi(z_2)$ must have length $\|z_1\|^{-1}$ times a multiple of $p$. Also, for each $\pi(z_2)$, there exist $O(p^{1/2}/\|z_1\|)$ points in $\mathcal P$ that project to $\pi(z_2)$. Thus there exist $$O\left(\frac{p^{\frac{1}{2}}}{p\|z_1\|^{-1}} \cdot \frac{p^{\frac{1}{2}}}{\|z_1\|}\right) = O(1)$$
candidates for $z_2$ contributing nontrivially to the sum. To recapitulate, there exist $O(p)$ contributing pairs $(z_1,z_2)$ such that $z_1$ is primitive. Similarly, for each $j \geq 1$, there exist $O(j^{-2n}p)$ contributing pairs $(z_1,z_2)$ such that $z_1$ is $j$ times a primitive vector. Therefore the above sum vanishes in the $p$-limit, as desired.

\subsection{Conclusion}

To recap, we obtained that the $p \rightarrow \infty$ limit of \eqref{eq:rk_1} is
\begin{align*}
\frac{1}{\zeta(2n)}\sum_{\substack{j \geq 1\\ k \neq 0}} \int_{\R^{2n}} f(jz,kz)dz.
\end{align*}
Note that this term has an identical counterpart in Rogers' integration formula in the case $k = 2$ \cite{Rog55}.

\section{Contribution from $\Delta_{n,p}$, $\F_p$-rank 2}

\subsection{A reduction}

Here we consider the following part of the sum \eqref{eq:toest_main}:
\begin{align*}
\frac{1}{p^{n(n+1)/2}}\sum_{B} \sum_{\substack{z_1,z_2 \in \Z^{2n}\setminus \{0\}\\ \mathrm{rk}_p\{x_1,x_2\}=2}} f(p^{-\frac{1}{2}}(x_1;x_1B+py_1)g,p^{-\frac{1}{2}}(x_2;x_2B+py_2)g).
\end{align*}
Applying Proposition \ref{prop:eqdist} with $k=2$, this becomes
\begin{align*}
\frac{1}{p^{2n-1}}\sum_{\substack{x_1,x_2 \in \Z^{n}\\ \mathrm{rk}_p=2}}
\sum_{\substack{y_1,y_2\in\Z^n \\ (x_1;y_1)J(x_2;y_2)^\intercal  \equiv_p 0}} f(p^{-\frac{1}{2}}z_1g,p^{-\frac{1}{2}}z_2g).
\end{align*}
We will prove that 
\begin{align} \label{eq:rk2_temp1}
\frac{1}{p^{2n-1}}\sum_{\substack{x_1,x_2 \in \Z^{n}\\ \mathrm{rk}_p\leq 1}}\sum_{\substack{y_1,y_2\in\Z^n \\ (x_1;y_1)J(x_2;y_2)^\intercal  \equiv_p 0}} f(p^{-\frac{1}{2}}z_1g,p^{-\frac{1}{2}}z_2g)
\end{align}
vanishes as $p \rightarrow \infty$, so that we could instead work with
\begin{align} \label{eq:rk2}
\frac{1}{p^{2n-1}}\sum_{\substack{z_1\neq 0,z_2\in\Z^{2n} \\ z_1Jz_2^\intercal  \equiv_p 0}} f(p^{-\frac{1}{2}}z_1g,p^{-\frac{1}{2}}z_2g).
\end{align}

In case $\mathrm{rk}_p\{x_1,x_2\} = 0$, it is easy to see that that part of \eqref{eq:rk2_temp1} amounts to $O(p^n/p^{2n-1})=O(p^{-n+1})$. So assume $\mathrm{rk}_p\{x_1,x_2\} = 1$. 
By arguing as in Sections \ref{sec:rk1_dep} and \ref{sec:rk1_indep}, one can show that there exist $O(p^{n/2})$ $\F_p$-rank $1$ pairs $(x_1,x_2)$ contributing nontrivially to \eqref{eq:rk2_temp1}. Fixing such $(x_1,x_2)$, the number of contributing $(y_1,y_2)$ is $O(p^{n-1/2})$, since the entries of $y_1,y_2$ are each of size $O(p^{1/2})$ and they are subject to a nontrivial linear constraint. Hence \eqref{eq:rk2_temp1} has size $O(p^{(3n-1)/2}/p^{2n-1}) = O(p^{-(n-1)/2})$, which for $n\geq 2$ vanishes as $p\rightarrow\infty$, as desired.

In what follows, we will evaluate \eqref{eq:rk2} in two different ways, leading to two different formulas. It turns out that they indeed agree with each other.

\subsection{By counting formulas}

It suffices to assume $f(z_1,z_2)=\mathbf 1_{B_1}(z_1)\mathbf 1_{B_2}(z_2)$ for fixed balls $B_1,B_2\subseteq \R^{2n}$, since any integrable $f$ can be approximated in the $L^1$-norm arbitrarily well by finitely many functions of this form.
Start by rewriting \eqref{eq:rk2} as
\begin{align}\label{eq:rk2_2}
\frac{1}{p^{2n-1}}\sum_{j\geq1}\sum_{z_1\in\Z^{2n}_\pr}\mathbf 1_{B_1}(p^{-\frac{1}{2}}jz_1g)\sum_{k \in \Z}\sum_{\substack{z_2 \in \Z^{2n} \\ z_1Jz_2^\intercal =kp}} \mathbf 1_{B_2}(p^{-\frac{1}{2}}z_2g).
\end{align}

Write $Z = \spn_\Z\langle z_1gJ\rangle$, a rank $1$ sublattice of $\Z^{2n}gJ$. The symplecticity of $g$ implies that the polar lattice of $\Z^{2n}gJ$ is $\Z^{2n}g$: since $g^\intercal Jg = J$, $(gJ)^{-\intercal}=g^{-\intercal}J=Jg$, and clearly $\Z^{2n}Jg = \Z^{2n}g$.
Thus $Z^\perp$ is a rank $2n-1$ sublattice of $\Z^{2n}g$. Now define 
\begin{align*}
L(k,p^{-\frac{1}{2}}z_1g) := \spn_\R (Z^\perp) + \frac{kp^{-\frac{1}{2}}z_1gJ}{\|p^{-\frac{1}{2}}z_1g\|^2}.
\end{align*}
The innermost sum of \eqref{eq:rk2_2} counts the number of points of $p^{-1/2}\Z^{2n}g$ contained in
\begin{align*}
B'_2:= L(k,p^{-\frac{1}{2}}z_1g)\cap B_2,
\end{align*}
or equivalently, of the points of $\Z^{2n}g$ contained in 
\begin{align*}
p^\frac{1}{2}B'_2 = \left(\spn_\R (Z^\perp) + \frac{kpz_1gJ}{\|z_1g\|^2}\right)\cap p^\frac{1}{2}B_2.
\end{align*}

It is easy to check that $\Z^{2n}g \cap p^{1/2}L(k,p^{-1/2}z_1g)$ has full rank. 
Well-known lattice-point counting formulas --- e.g., \cite[Lemma 3.6]{Kim24}, or \cite[Theorem 5.4]{Wid10} --- imply that the quantity is
\begin{align}\label{eq:counting_formulas_inner_sum}
&\frac{\mathrm{vol}(p^\frac{1}{2}B_2')}{\det Z^\perp} + O_n\left(\sum_{i=0}^{2n-2}\frac{p^\frac{i}{2}}{\lambda_1(Z^\perp)\cdots\lambda_i(Z^\perp)}\right) \nonumber \\
&= \frac{p^{n-1}}{\|p^{-\frac{1}{2}}z_1g\|}\mathrm{vol}(B_2') +  O_n\left(\sum_{i=0}^{2n-2}\frac{p^\frac{i}{2}}{\lambda_1(Z^\perp)\cdots\lambda_i(Z^\perp)}\right),
\end{align}
where $\lambda_i$ is the $i$'th successive minimum. 

Later in our computation, we will need the intended error terms above to be $o_g(p^{n-1})$. Since $\lambda_i(Z^\perp) \geq \lambda_i(\Z^{2n}g)$, this is easily seen to hold for $i < 2n-2$. For $i=2n-2$, however, this is not enough. Worse, if $\lambda_1(Z^\perp)\cdots\lambda_{2n-2}(Z^\perp)$ is small, which is very possible, the intended error will be of size comparable to $p^{n-1}$.

Let us first give a short informal discussion of how we plan to get around this difficulty. The idea is that, since $\det Z^\perp = \|z_1gJ\|$, which is proportional to $p^{1/2}$ most of the time, we expect $\lambda_j(Z^\perp) \approx p^{1/(2(2n-1))}$ most of the time for any $j=1,\ldots,2n-1$. To see this, first recall the fact (e.g.,\cite{HK23}) that the shape of $Z^\perp$ becomes asymptotically equidistributed in $\xlat_{2n-1}$ for uniform choices of $Z$ of covolume $\leq j^{-1}p^{1/2}$ as $p\rightarrow\infty$. For such equidistributed lattices, recall another well-known fact that, for most choices of $Z$, $\lambda_1$ is a constant times the $1/(2n-1)$'th power of its determinant (e.g.,\cite[Theorem 6]{LN25}). Hence we expect $\lambda_1(Z^\perp) \approx p^{1/(2(2n-1))}$. Since $\lambda_1\leq\lambda_2\leq\lambda_3\leq\ldots$ and by Minkowski's second theorem, other $\lambda_j$ must be of the same growth order in $p$. 

What we actually need is in fact a lot weaker: for any $\delta > 0$, it is sufficient that $\lambda_1(Z^\perp)\cdots\lambda_{2n-2}(Z^\perp) \gg p^\delta$ most of the time, or equivalently by Minkowski's second theorem, $\lambda_{2n-1}(Z^\perp) \ll p^{1/2-\delta}$.

We now give a rigorous argument for this step. Rewrite the problematic error term as
\begin{align*}
\frac{p^{n-1}}{\lambda_1(Z^\perp)\cdots\lambda_{2n-2}(Z^\perp)} &= O\left(p^{n-1}\frac{\lambda_{2n-1}(Z^\perp)}{\det Z^\perp}\right) \\
&= O\left(p^{n-1}\frac{\det Z^\pi}{\lambda_1(Z^\pi)}\right),
\end{align*}
where we wrote $Z^\pi = \Z^{2n}gJ/Z$; first line follows by Minkowski's second theorem, second by the transference theorem (see e.g., \cite{Ban93}; it may also be helpful to recall that $Z^\pi$ can be identified with $(Z^\perp)^P$, as explained in Section \ref{subsect:polaretc}). 
\begin{lemma} \label{lemma:control}
Let $A>0$, and let $p$ be sufficiently large. There exist constants $0<\delta,\eta<1$ such that the number of $z_1gJ \in \Z^{2n}gJ$ of length at most $Ap^{1/2}$ such that $\lambda_1(Z^\pi) < (\det Z^\pi)^{1-\delta}$ is at most $O(p^{n-0.49\eta})$.
\end{lemma}
\begin{proof}
Following \cite{HK23}, for an element $\tilde g \in \SL(2n,\R)$, we consider the \emph{refined Iwasawa} decomposition\footnote{Here and in the rest of the proof, we take the transpose of how the matrices are written in \cite{HK23}, since we are modding out by $\SL(2n,\Z)$ from the left.} $\tilde g = na'_ta''_wk$, where $n$ is a lower-triangular unipotent matrix, $k \in \SO(2n,\R)$, and
\begin{align*}
a' = \begin{pmatrix} e^t &  \\ & \diag(\underbracket{e^{-\frac{t}{2n-1}},\ldots,e^{-\frac{t}{2n-1}}}_{n-1\ \mathrm{times}}) \end{pmatrix},
a'' = \begin{pmatrix}
1 & \\
& \diag( e^{-\frac{w_1}{2}},e^{\frac{w_1-w_2}{2}}, \ldots, e^{\frac{w_{2n-2}}{2}}) \\
\end{pmatrix}.
\end{align*}

Suppose that $\tilde g$ is an element of the fundamental domain (as chosen in \cite[(6.1)]{HK23}) of 
\begin{align*}
\Gamma' = \begin{pmatrix} 1 &  \\ \Z^{2n-1}& \SL(2n-1,\Z) \end{pmatrix} \subseteq \SL(2n,\R),
\end{align*}
and that the row vectors of $\tilde g$ span $\Z^{2n}gJ$ and its first row equals $z_1gJ$.
Then by \cite[Prop. 4.3]{HK23}, we have $e^t = \|z_1g\|=\|z_1gJ\|$, and in addition by the reduction theory of $\SL(*,\R)$ (\emph{cf.} \cite[Lemma 2.5]{Bor19}) we also have
\begin{align*}
\lambda_1(Z^\pi) = \Theta(e^{-\frac{t}{2n-1}-\frac{w_1}{2}}),\  \lambda_1(Z^\pi)\lambda_2(Z^\pi)=\Theta(e^{-\frac{t}{2n-1}-\frac{w_2}{2}}), \ldots,
\end{align*}
and so on (recall $f=\Theta(g)$ means $f=O(g)$ and $g=O(f)$). Thus $\lambda_1(Z^\pi) < (\det Z^\pi)^{1-\delta}$ if and only if $w_1 = \eta t + O(1)$, where $\eta$ satisfies
\begin{align*}
\frac{1}{2n-1}-\frac{\eta}{2}=1-\delta.
\end{align*}
It is easy to come up with a pair $0 < \delta,\eta < 1$ satisfying this condition.

By copying the proof of \cite[Prop. 8.3]{HK23} almost word by word (also \emph{cf.} Ibid., Prop. 6.1), except replacing $\SL(2n,\Z)$ by $\SL(2n,\Z)g$, one obtains that the number of $z_1gJ \in \Z^{2n}gJ$ of length at most $Ap^{1/2}$ such that $\lambda_1(Z^\pi) < (\det Z^\pi)^{1-\delta}$ is at most $O(p^{n-0.49\eta})$, as desired.
\end{proof}

With the above lemma, we return to the consideration of \eqref{eq:rk2_2}. Let us temporarily say $z_1 \in \Z^{2n}_\pr$ is \emph{normal} if $\|z_1g\| > p^{0.49}$ and $\lambda_1(Z^\pi) \geq (\det Z^\pi)^{1-\delta}$. By Lemma \ref{lemma:control} and \eqref{eq:counting_formulas_inner_sum}, the contribution to \eqref{eq:rk2_2} of not-normal $z_1$'s is $O(p^{-0.49\eta} + p^{-0.02n})$, which vanishes as $p\rightarrow\infty$. 
The contribution of normal $z_1$'s is, up to terms vanishing in the $p$ limit,
\begin{align*}
\frac{1}{p^n}\sum_{j\geq1}\sum_{\substack{z_1\in\Z^{2n}_\pr\\\mathrm{normal}}} \mathbf 1_{B_1}(p^{-\frac{1}{2}}jz_1g)\frac{1}{\|p^{-\frac{1}{2}}z_1g\|}\sum_{k\in\Z}\mathrm{vol}(L(k,p^{-\frac{1}{2}}z_1g)\cap B_2).
\end{align*}

For each $z_1$, there are $O(\|p^{-1/2}z_1g\|) = O_g (1)$ values of $k \in \Z$ that contribute nontrivially, and each term is bounded. Thus the above sum converges, and it is asymptotically equivalent to
\begin{align*}
\frac{1}{p^n}\sum_{j\geq1}\sum_{z_1\in\Z^{2n}_\pr} \mathbf 1_{B_1}(p^{-\frac{1}{2}}jz_1g)\frac{1}{\|p^{-\frac{1}{2}}z_1g\|}\sum_{k\in\Z}\mathrm{vol}(L(k,p^{-\frac{1}{2}}z_1g)\cap B_2).
\end{align*}
As a Riemann sum, this converges as $p \rightarrow \infty$ to
\begin{align*}
\frac{1}{\zeta(2n)}\sum_{\substack{j\geq 1 \\ k \in \Z}}\int_{\R^{2n}}\mathbf 1_{B_1}(jz_1)\mathrm{vol}(L(k,z_1)\cap B_2)\frac{dz_1}{\|z_1\|},
\end{align*}
or more suggestively,
\begin{align}\label{eq:mainterm_1}
\frac{1}{\zeta(2n)}
\sum_{\substack{j\geq 1 \\ k \in \Z}}
\int_{\R^{2n}}\int_{L(k,z_1)}f(jz_1,z_2)dz_2\frac{dz_1}{\|z_1\|}.
\end{align}

\subsection{By Fourier analysis} We will now briefly sketch how it is possible to arrive at the ``symplectic'' term in Theorem \ref{thm:main-Schwartz} directly from the rank-two expression \eqref{eq:rk2}. (The argument that follows is not logically relevant to the proof of Theorem \ref{thm:main-Schwartz}; rather, this theorem will be deduced from Theorem \ref{thm:main} in the next section.)\\[+0.5em]
For argument's sake, we assume that $f(z_1,z_2) = f_1(z_1)f_2(z_2)$ factors as a product of two compactly supported Schwartz functions $f_1,f_2$ on $\R^{2n}$. In order to handle the condition $z_1 J z_2^\intercal \equiv_p 0,$ we introduce the lattice $\Lambda(z_1) := \{z_2 \in \Z^{2n} \mid z_1Jz_2^\intercal  \equiv_p 0\}$ and rewrite the expression \eqref{eq:rk2} as
\begin{align*}
\frac{1}{p^{2n-1}}\sum_{j \geq 1}\sum_{z_1\in\Z^{2n}_\pr}f_1(p^{-1/2}Jz_1g)\sum_{z_2\in\Lambda(z_1)}f_2(p^{-1/2}z_2g).
\end{align*}
The idea is now to apply the Poisson summation formula to the innermost sum. \\[+0.5em]
First, we observe that the lattice dual to $\Lambda(z_1)$ is $\Lambda(z_1)^P = \Z^{2n} + p^{-1}\Z z_1 J$. (Namely, the right-hand side is indeed contained in the left-hand side, and has the correct determinant $p^{-1} = (\det \Lambda(z_1))^{-1}$.) Next, we use the well-known properties of the Fourier transform to compute that
\begin{align*}
    \left[ \mathcal{F} \! \left( z_2\mapsto f_2(p^{-1/2}z_2g)\right) \right]\!(\xi_2)
    =
    p^n\widehat{f}_2(p^{1/2}\xi_2g^{-\intercal}).
\end{align*}
Letting $q = -pmJ + k z_1$, we can now write the sum over $z_2$ as
\begin{align*}
&p^{n-1}\sum_{m \in \Z^{2n}}\sum_{k \in (-\frac{p}{2},\frac{p}{2})} 
\widehat f_2\bigl( p^\frac{1}{2}(m+p^{-1}kz_1J)g^{-\intercal}\bigr)
=
p^{n-1}\sum_{m \in \Z^{2n}}\sum_{k \in (-\frac{p}{2},\frac{p}{2})} 
\widehat f_2\bigl( p^{-\frac{1}{2}}q J g^{-\intercal} \bigr).
\end{align*}
\par Since $q \equiv_p k z_1,$ we see that the $\F_p$-rank of $\lbrace q, z_1 \rbrace$ is $1.$ Because $\smash{\widehat{f}_2}$ has rapid decay, we can restrict attention to those $q$ with $\Vert q \Vert \ll \smash{\sqrt{p}}$; and we can then argue (just as in Section \ref{sec:rk1_indep}) that all but a negligible proportion of the pairs $(q, z_1)$ will be linearly dependent over $\Q.$ As $z_1$ is primitive, this means that $q = k z_1$ for some $k \in \Z$. Using that $Jg^{-\intercal}=gJ$, we find that \eqref{eq:rk2} is asymptotic to
\begin{align*}
\frac{1}{p^n}
\sum_{j \geq 1}
\sum_{k \in \Z}
\sum_{z_1 \in \Z^{2n}_\pr}
f_1(p^{-\frac{1}{2}}j z_1 g)\widehat f_2(p^{-\frac{1}{2}} k z_1 g J).
\end{align*}
This is now a "primitive" Riemann sum with step length $p^{-1/2}$ in each of the $2n$ coordinates of $z_1$. As in Section \ref{sec:rk1_dep}, because the primitive lattice points have density $1/\zeta(2n)$ in $\Z^{2n}$, we see that this sum converges to
\begin{align} \label{eq:mainterm_2}
\frac{1}{\zeta(2n)}\sum_{\substack{j\geq 1 \\k \in \Z}}\int_{\R^{2n}} f_1(jz)\widehat f_2(kzJ) \, dz
\end{align}
as $p \rightarrow \infty.$

\subsection{Proof of Theorem \ref{thm:main-Schwartz} based on Theorem \ref{thm:main}}\label{subsect:relation_between_formulas} $\ $\\ 
We now reconcile the two seemingly different main terms \eqref{eq:mainterm_1} and \eqref{eq:mainterm_2}, and so deduce Theorem \ref{thm:main-Schwartz} from Theorem \ref{thm:main}, using the theory of the Radon transform. (For more details on this topic, we refer the reader to the book \cite{ehrenpreis}.) 

For a unit vector $v \in \R^{2n}$ and $t \in \R,$ consider the hyperplane $\Pi (v, t) \subseteq \R^{2n}$ perpendicular to $v$ at distance $|t|$ from the origin. The \textit{Radon transform} of an (absolutely) integrable function such as $f_2$ is the function of $(v,t)$ given by the integral
\begin{align*}
    (\radon  f_2)(v,t) 
    = \int_{\Pi (v,t)} f_2(x) \, d_\Pi (x) 
    = \int_{\Pi (v,t)} f_2(x) \, dx
\end{align*}
where $d_\Pi(x) = dx$ is Lebesgue measure on the hyperplane $\Pi(v, t).$ 

As in the introduction, we consider the specific hyperplane
\begin{align*}
    L(k, z_1) 
    = 
    \Pi(z_1 J, k) 
    =
    \Pi \! \left( 
        \frac{z_1 J}{\Vert z_1 \Vert}, 
        \frac{k}{\Vert z_1 \Vert} 
        \right)
\end{align*}
and note that the right-hand side of Theorem \ref{thm:main} involves the integral
\begin{align}\label{radon_reconciliation}
\frac{1}{\zeta(2n)} 
\sum_{j \geq 1} 
\sum_{k \in \mathbb{Z}}
\int_{\mathbb{R}^{2n}} 
(\radon f_2)\left( 
\frac{z_1 J}{\Vert z_1 \Vert}, 
\frac{k}{\Vert z_1 \Vert} \right)
f_1(j z_1) \, \Vert z_1 \Vert^{-1} \, d z_1.
\end{align}
It will be convenient to write $u := z_1 J / \Vert z_1 \Vert$ and $p := k / \Vert z_1 \Vert$, and 
\begin{align*}
\varphi_u(k) := (\radon f_2)(u, k/\Vert z_1 \Vert) = (\radon f_2)(u, p)
\end{align*}
for the Radon transform of $f_2$. \\ \par 

We will first assume that $f_1$ and $f_2$ are Schwartz functions and show the equivalence of Thm. \ref{thm:main} and Thm. \ref{thm:main-Schwartz} under this assumption. This will prove the first part of Thm. \ref{thm:main-Schwartz} that concerns functions in the Schwartz class. 

By interchanging the sum over $k$ and the integral, we see that \eqref{radon_reconciliation} contains the sum
\begin{align} \label{radon_and_poisson}
\sum_{k \in \mathbb{Z}} (\mathcal{R} f_2)\left( \frac{z_1 J}{\Vert z_1 \Vert}, 
\frac{k}{\Vert z_1 \Vert}\right) 
=
\sum_{k \in \mathbb{Z}} \varphi_u (k/\Vert z_1 \Vert)
= 
\Vert z_1 \Vert \sum_{\xi \in \mathbb{Z}} \widehat{\varphi}_u (\Vert z_1 \Vert \xi),
\end{align}
where the final equality is due to the Poisson summation formula. 

Computing the Fourier transform with the Fourier slice theorem, we find that
\begin{align}\label{big-ol-slice}
\Vert z_1 \Vert \, \widehat{\varphi}_u (\Vert z_1 \Vert \xi)
= 
\Vert z_1 \Vert \, \smash{\widehat{f}}_2 (\Vert z_1 \Vert \xi u) 
= 
\Vert z_1 \Vert
\smash{\widehat{f}}_2 \biggl(
\frac{\Vert z_1 \Vert \xi z_1 J}{\Vert z_1 \Vert} \biggr)
=
\Vert z_1 \Vert
\smash{\widehat{f}}_2 \bigl(
\xi z_1 J
\bigr).
\end{align}
Inserting this into \eqref{radon_and_poisson} and inserting the result into \eqref{radon_reconciliation}, we conclude that the latter expression is
\begin{align*}
\frac{1}{\zeta(2n)} 
\sum_{j \geq 1} 
\sum_{\xi \in \mathbb{Z}}  
\int_{\mathbb{R}^{2n}}
\smash{\widehat{f}}_2 \bigl( \xi x J \bigr)
f_1(j x) \, d x .
\end{align*}
This shows that \eqref{eq:mainterm_1} and \eqref{eq:mainterm_2} are equal. \\ \par 

We now extend the argument to cover the other cases announced in Theorem \ref{thm:main-Schwartz}. thus completing the proof of this theorem. Since this task only involves $f_2$, we will drop the subscript and write $f := f_2$ for simplicity.

The Radon transform in \eqref{radon_reconciliation} appears in the integrand. Therefore we have to ensure that $\mathcal{R}f_2 = \mathcal{R}f$ is well-defined for almost every hyperplane $\Pi$. To do this, we treat the two cases separately: First, if $f$ is Riemann integrable, let $D(f)$ be the set of all points of discontinuity of $f$. As this is a nullset, so is the set of all $z_1 \in \R^{2n}$ and $p \geq 0$ such that $\Pi (z_1 J / \Vert z_1 \Vert, p) \cap D(f)$ has positive measure. Therefore $f|_\Pi$ is integrable for almost every hyperplane $\Pi.$ Next, if $f$ is a bounded, Borel measurable function with compact support, then it is immediate that $f|_\Pi$ has the same properties for any hyperplane $\Pi$; hence $\mathcal{R}f$ is a well-defined function in this case as well.\par

Let $\varepsilon > 0$ and replace $f$ with the mollified function $f \ast \Phi_\varepsilon$, where $\Phi_\varepsilon (x) = \varepsilon^{-2n} \varphi(\varepsilon^{-1} x)$ and $\Phi$ denotes the smooth bump function
\begin{align*}
    \Phi (x) = C \cdot {\bf 1} \! \bigl[ \Vert x \Vert < 1 \bigr] \exp \left( - \frac{1}{1 - \Vert x \Vert^2}\right) \qquad (\text{for } x \in \R^{2n}) 
\end{align*}
where the constant $C > 0$ ensures that $\Phi$ has integral $1.$ 
Since $f_\varepsilon$ is a Schwartz function, the argument above goes through with $f_\varepsilon$ in place of $f_2$, and the equivalent of \eqref{radon_and_poisson} and \eqref{big-ol-slice} becomes
\begin{align}\label{mollified-Radon-Poisson}
    \sum_{k \in \mathbb{Z}} (\mathcal{R} f_\varepsilon)\left( \frac{z_1 J}{\Vert z_1 \Vert}, 
    \frac{k}{\Vert z_1 \Vert}\right) 
    &= \nonumber
    \Vert z_1 \Vert \sum_{\xi \in \mathbb{Z}} 
    \widehat{f} \bigl( \xi z_1 J \bigr) \widehat{\Phi_\varepsilon} \bigl( \xi z_1 J \bigr)
    \\ &=
    \Vert z_1 \Vert \sum_{\xi \in \mathbb{Z}} 
    \widehat{f} \bigl( \xi z_1 J \bigr) \widehat{\Phi} \bigl( \varepsilon \xi z_1 J \bigr).
\end{align}
Our goal is now to let $\varepsilon \rightarrow 0$ on both sides of \eqref{mollified-Radon-Poisson} and argue that the desired statement emerges in the limit. \par 
We claim that the dominated convergence theorem applies to the limit on the left-hand side. Indeed, for any $\varepsilon \in (0,1)$, the Radon transform $\mathcal{R} f_\varepsilon$ is uniformly dominated by the Radon transform of a slightly larger function supported on a thickening of the compact support of $f$. Therefore the left-hand side above is dominated by a series over $k \in \Z$ which is convergent because it is comparable to the right-hand side of \eqref{radon_and_poisson}.\par 
Concerning the right-hand side, we note that
\begin{align*}
    \widehat{f}(\xi z_1 J) \smash{\widehat{\Phi}}(\varepsilon \xi z_1 J) \longrightarrow \smash{\widehat{f}}(\xi z_1 J) \smash{\widehat{\Phi}}(0)
    = \smash{\widehat{f}}(\xi z_1 J)
\end{align*}
when $\varepsilon \rightarrow 0$, because of the normalization of $\Phi$ discussed above. Therefore the dominated convergence theorem will apply, provided that we can show that
\begin{align}\label{fhat_series_almost_surely_convergent}
    \sum_{\xi \in \Z} \Big| \widehat{f} (\xi z_1 J ) \Big| < \infty 
    \qquad (\text{for almost every } z_1 \in \R^{2n}).
\end{align}
To prove this, we can average this series over an annulus $A \subseteq \R^{2n}$ centered at the origin, and argue as follows: Tonnelli's theorem and Cauchy--Schwarz give
\begin{align*}
    \int_A \sum_{\xi \in \mathbb{Z}} 
    \Big| \smash{\widehat{f}} (\xi z_1 J ) \Big| \, d z_1
    &= 
    \sum_{\xi \in \mathbb{Z}}  \int_A 
    \Big| \smash{\widehat{f}} (\xi z_1 J) \Big| \, d z_1
    \\ &\leq
    \smash{\widehat{f}} (0)\operatorname{Vol}(A) + 
    2 \sum_{\xi = 1}^\infty  \frac{1}{\xi^{2n}} \int_{\xi A}
    \Big| \smash{\widehat{f}} (z_1 J) \Big| \, d z_1
    \\ &\leq
    \smash{\widehat{f}} (0) \operatorname{Vol}(A) + 
    2 \sum_{\xi = 1}^\infty  \frac{1}{\xi^{2n}} 
    \operatorname{Vol}(\xi A)^{1/2} \, \big\| \smash{\widehat{f}} \, \big\|_{L^2}
    \\ &\ll_{A}
    \big\| f \big\|_{L^2} \sum_{\xi = 1}^\infty  \frac{1}{\xi^{n}},
\end{align*}
where we also used Plancherel's theorem. Since $A$ was arbitrary, \eqref{fhat_series_almost_surely_convergent} follows.\par 
Finally, taking the limit inside the sums in \eqref{mollified-Radon-Poisson}, we arrive at the statement
\begin{align*}
    \sum_{k \in \mathbb{Z}} (\mathcal{R} f)\left( \frac{z_1 J}{\Vert z_1 \Vert}, 
    \frac{k}{\Vert z_1 \Vert}\right) 
    =
    \Vert z_1 \Vert \sum_{\xi \in \mathbb{Z}} 
    \widehat{f} \bigl( \xi z_1 J \bigr).
\end{align*}
Inserting this into \eqref{radon_reconciliation}, we see that \ref{eq:mainterm_1} and \ref{eq:mainterm_2} are equal. \\ \par 
In addition to showing the equality of the two seemingly different terms on the ``Euclidean'' sides of the moment formula, the above computation reconciles Theorem \ref{thm:main} and Theorem \ref{thm:main-Schwartz} entirely, as the term corresponding to $\xi = 0$ is easily seen to equal $\left( \int f_1 \, dz \right) \left( \int f_2 \, dz \right)$. The proof of Thm. \ref{thm:main-Schwartz} from Thm. \ref{thm:main} is thus concluded.

\section{Contribution from $\Delta_{i,p}$ for $i < n$}
\label{sec:Delta_i}

The goal of this section is to prove the following, which would complete our proof of the main result Theorems \ref{thm:main} and \ref{thm:main-Schwartz}.
\begin{proposition}\label{prop:error}
For $i=0,\ldots,n-1$ and $F$ as in \eqref{eq:2strans},
\begin{align*}
\frac{1}{|\Delta_p|}\sum_{h \in \Delta_{i,p}} F(p^{-\frac{1}{2}}hg)
\end{align*}
vanishes as $p \rightarrow \infty$. 
\end{proposition}
Note that in the case of a bounded function (which our $F$ is not), the claim in Proposition \ref{prop:error} would be very straightforward to prove. Rather, the substance of this result concerns the case $F \in L^1 (\Gamma \backslash G).$

By unraveling the definitions, and making harmless substitutions --- e.g. replacing $f\circ g$ with the characteristic function of a box --- it suffices to work with
\begin{align}\label{eq:error}
\frac{1}{p^{n(n+1)/2}}\sum_{h}\sum_{z_1,z_2\in\Z^{2n}\setminus \{0\}}f(p^{-\frac{1}{2}}z_1h,p^{-\frac{1}{2}}z_2h),
\end{align}
where $h \in \Delta_{i,p}$ is of the form
\begin{align*}
\begin{pmatrix}
pD^{-\intercal} & B \\ & D
\end{pmatrix}, \mbox{ where }
D=
\begin{pmatrix}
I_{n-i} & \tilde D \\ & pI_i
\end{pmatrix},
B =
\begin{pmatrix}
0 & 0 \\ 0 & B_{22}
\end{pmatrix},
\end{align*}
as described in Lemma \ref{lemma:reps}; here, recall $\tilde D$ is an $(n-i)\times i$ matrix, $B_{22}$ is an $i \times i$ symmetric matrix, and both have entries in $\{0,1,\ldots,p-1\}$.

Write $z_1 = (x'_1,x_1;y_1,y_1')$, where $x'_1,y_1\in\Z^{n-i}$ and $x_1,y'_1\in\Z^i$. Then
\begin{align}\label{eq:errexp_z1}
z_1h=(px'_1-x_1\tilde D^\intercal ,x_1;y_1,x_1B_{22}+y_1\tilde D+py'_1).
\end{align}
Notice that for each choice of $x_1$ and $y_1$, there exist at most unique $x'_1$ and $y'_1$ such that $z_1$ contributes nontrivially to the sum above. 
Likewise, we also have
\begin{align}\label{eq:errexp_z2}
z_2h=(px'_2-x_2\tilde D^\intercal ,x_2;y_2,x_2B_{22}+y_2\tilde D+py'_2),
\end{align}
and the same comment applies.

\subsection{Preliminary analysis}

Consider \eqref{eq:error} with \eqref{eq:errexp_z1},\eqref{eq:errexp_z2} plugged in. By what we just discussed, we can crudely bound the inner sum as a sum over $x_1,y_1$ and $x_2,y_2$, which gives $O(p^n)$ for the inner sum. Also, the sum over $h$ has $p^{i(n-i)+i(i+1)/2}=p^{i(n+1/2)-i^2/2}$ terms, so \eqref{eq:error} can be bounded by
\begin{align*}
p^{i(n+\frac{1}{2})-\frac{i^2}{2}+n-\frac{n(n+1)}{2}} = p^{i(n-\frac{i-1}{2})-\frac{(n-1)n}{2}}.
\end{align*}
In particular, if $i=0$, then \eqref{eq:error} vanishes as $p \rightarrow \infty$. However, some terms are not covered by this naive estimate, and call for the more precise analysis below.

\begin{proposition} \label{prop:eqdist_err}
Suppose $1 \leq i < n$, fix $x_1,x_2\in\F_p^i, y_1,y_2\in\F_p^{n-i}$, and consider the map $T_{x,y}:\F_p^{i(n-i)+i(i+1)/2} \rightarrow \F_p^{2n}$ by
\begin{align*}
T_{x,y}(\tilde D,B_{22}) = 
(-x_1\tilde D^\intercal , x_1B_{22}+y_1\tilde D, -x_2\tilde D^\intercal , x_2B_{22}+y_2\tilde D).
\end{align*}
Let us also write
\begin{align*}
S_{x,y}:=\{(\tilde x_1,\tilde y_1,\tilde x_2, \tilde y_2)\subseteq\F_p^{2n}: (\tilde x_1,x_1;y_1,\tilde y_1)J(\tilde x_2,x_2;y_2,\tilde y_2) \equiv 0\ (\mathrm{mod}\ p)\}.
\end{align*}
Here $\tilde x_1,\tilde x_2\in\F_p^{n-i}$ and $\tilde y_1,\tilde y_2 \in \F_p^i$.

\begin{enumerate}[(a)]
\item Suppose $\mathrm{rk}_p\{x_1,x_2\}=2$. Then $T_{x,y}$ is a surjection onto $S_{x,y}$, and $n(n+1)/2-\dim\ker T_{x,y} > 2n-1$.

\item Suppose $\mathrm{rk}_p\{x_1,x_2\}=1$ and $\rk_p\{y_1,y_2\}=2$. Then $T_{x,y}$ surjects onto the subset of $S_{x,y}$ such that $\mathrm{rk}_p\{(\tilde x_1, x_1),(\tilde x_2, x_2)\}=1$. 
Also, $n(n+1)/2-\dim\ker T_{x,y} > 2n-1$.

\item Suppose $\mathrm{rk}_p\{x_1,x_2\}=1$ and $\rk_p\{y_1,y_2\}=1$, but $(x_1,y_1)$ and $(x_2,y_2)$ are linearly independent. Then $T_{x,y}$ surjects onto the subset of $S_{x,y}$ such that $\mathrm{rk}_p\{(\tilde x_1, x_1),(\tilde x_2, x_2)\}=1$. 
Also, $n(n+1)/2-\dim\ker T_{x,y} \geq 2n-1$.

\item Suppose $\mathrm{rk}_p\{x_1,x_2\}=1$, $(x_1,y_1)$ and $(x_2,y_2)$ are linearly dependent. Then $T_{x,y}$ surjects onto the subset of $S_{x,y}$ such that $(\tilde x_2,x_2;y_2,\tilde y_2) = k(\tilde x_1,x_1;y_1,\tilde y_1)$ for some $k \in\F_p$ or $(\tilde x_1,x_1;y_1,\tilde y_1)=0$. In this case $n(n+1)/2-\dim\ker T_{x,y} > n$.

\item Suppose $\rk_p\{x_1,x_2\}=0$. Then $T_{x,y}$ maps onto the subset of $S_{x,y}$ consisting of the elements of the form $(0,\tilde y_1,0,\tilde y_2)$. Consequently $n(n+1)/2-\dim\ker T_{x,y} \geq 2n-1$ or $n$ depending on whether $\rk_p\{y_1,y_2\}=2$ or $1$, respectively.

\end{enumerate}
\end{proposition}
\begin{proof}

(a) 
Suppose first $\rk_p\{x_1,x_2\}=\rk_p\{y_1,y_2\}=2$. By multiplying $\tilde D$ from the left by $\GL(n-i,\F_p)$ and from the right by $\GL(i,\F_p)$, and multiplying $B_{22}$ from the left by the same element of $\GL(i,\F_p)$, we may assume that $x_1=(1,0,0,\ldots)$, $x_2=(0,1,0,\ldots)$, and similarly with $y_1$ and $y_2$. From the construction of $T_{x,y}$ it can be straightforwardly checked that $T_{x,y}$ indeed maps into $S_{x,y}$. We need show that it also surjects.

Let $(\tilde x_1,\tilde y_1,\tilde x_2,\tilde y_2) \in S_{x,y}$. The first and second rows of $\tilde D^\intercal $ can be chosen so that $-x_1\tilde D^\intercal =\tilde x_1$ and $-x_2\tilde D^\intercal =\tilde x_2$. Also, the first row of $B_{22}$ and the first column of $\tilde D^\intercal $ can be chosen so that $x_1B_{22}+y_1\tilde D = \tilde y_1$. Though the first entry of $y_1D$ is forced to coincide with that of $x_1\tilde D^\intercal $, the first entry of $x_1B_{22}$ is unconstrained, so this is indeed possible. It remains to show that $x_2B_{22}+y_2\tilde D=\tilde y_2$ is solvable, under all the choices made so far. Observe that, if the second to last entries of these two vectors coincide, then the first automatically coincide as well, since both vectors are subject to the defining constraint 
\begin{align} \label{eq:condition}
(\tilde x_1,x_1;y_1,\tilde y_1)J(\tilde x_2,x_2;y_2,\tilde y_2) \equiv 0\ (\mathrm{mod}\ p)
\end{align}
of $S_{x,y}$, which forces what their first entries must be given the rest. Since the second row of $B_{22}$ (which equals $x_2B_{22}$) can be arbitrarily chosen except its first entry (which coincides with the second entry of $x_1B_{22}$), this can indeed be made to happen. The inequality on $\dim\ker T_{x,y}$ follows from the comparison with Proposition \ref{prop:eqdist}.

The same argument goes through even if $\rk_p\{y_1,y_2\}=1$ or $0$. This proves (a).

For (b)-(d), we have $\rk_p\{x_1,x_2\}=1$. Without loss of generality, we may assume $x_1=(1,0,\ldots)$ and $x_2=k_xx_1$ for some $k_x\in\F_p$.

(b) In this case, $\rk_p\{y_1,y_2\}=2$; again we may assume $y_1=(1,0,0,\ldots)$ and $y_2=(0,1,0,\ldots)$. Similarly as before, $T_{x,y}$ clearly maps into the said subset of $S_{x,y}$, and we need to show that it surjects. First we set the first row of $\tilde D^\intercal $ so that $-x_1\tilde D^\intercal =\tilde x_1$. This decides the first entries of $y_1\tilde D$ and $y_2\tilde D$, but by adjusting the other entries of $y_1\tilde D$ and the $(1,1)$-entry of $B_{22}$, we can achieve $x_1B_{22}+y_1\tilde D=\tilde y_1$. This in turn decides the first entry of $x_2B_{22}+y_2\tilde D$, but once we adjust the second to last entries of the second row of $\tilde D$ so that its second to last entries coincide with those of $\tilde y_2$, their first entries are also forced to match, since they are both determined by \eqref{eq:condition}. Therefore we can attain $x_2B_{22}+y_2\tilde D=\tilde y_2$. 

Since $\rk_p\{y_1,y_2\}=2$, $i \leq n-2$ must hold, and thus
\begin{align*}
\dim \ker T_{x,y} &= i(n-i)+i(i+1)/2 - (n+i-1) \\
&= (i-1)n-i^2/2-i/2+1 \\
&\leq (n-3)n-n^2/2+2n-2-n/2+2 \\
&= n^2/2-3n/2
\end{align*}
(notice the second line is maximized for $i=n-1/2$) from which it follows that $n(n+1)/2-\dim\ker T_{x,y} \geq 2n$ for $n \geq 2$.

(c) Next, suppose $\rk_p\{y_1,y_2\}=1$, and $y_2=k_yy_1$ but $k_y\neq k_x$. (The case $y_1\equiv_p0$ can be proved similarly and is omitted.) In this case, by solving for
\begin{align*}
\begin{pmatrix} 1 & 1 \\ k_x & k_y\end{pmatrix}
\begin{pmatrix} x_1B_{22} \\ y_1\tilde D \end{pmatrix}
=
\begin{pmatrix} \tilde y_1 \\ \tilde y_2 \end{pmatrix},
\end{align*}
we can determine the first row of $\tilde D$ and of $B_{22}$. It remains to determine the first column of $\tilde D$(= first row of $\tilde D^\intercal $) so that $-x_1\tilde D^\intercal  = \tilde x_1$. Again, by the same logic as earlier, this is automatically satisfied on the first entry, and the remaining entries of $\tilde D$ can be adjusted to make this true. 

$\dim \ker T_{x,y}$ is the same as in (b), except that $i=n-1$ is also possible. Thus
\begin{align*}
\dim\ker T_{x,y} &= i(n-i)+i(i+1)/2-(n+i-1) \\
&\leq (n-1) + (n-1)n/2 -2(n-1) \\
&= n^2/2-3n/2+1,
\end{align*}
and $n(n+1)/2 -\dim\ker T_{x,y} \geq 2n-1$.

(d) This time we suppose $k_y=k_x$. (The case $y_1\equiv_py_2\equiv_p0$ is similar and thus omitted.) We only need to worry about $\tilde x_1,\tilde y_1$. Fix the first row of $\tilde D^\intercal $ so that $-x_1\tilde D^\intercal =\tilde x_1$, and the first column so that $y_1\tilde D=\tilde y_1$ except on the first entry. Then the $(1,1)$-entry of $B_{22}$ can be set so that $x_1B_{22}+y_1\tilde D=\tilde y_1$. As for $\dim\ker T_{x,y}$,
\begin{align*}
\dim\ker T_{x,y} &= i(n-i)+i(i+1)/2-n \\
&\leq (n-1) + (n-1)n/2 -n \\
&= n^2/2-n/2-1
\end{align*}
and $n(n+1)/2-\dim\ker T_{x,y}\geq n+1$.

Finally, (e) is immediate from the fact that $T_{x,y} = (0,y_1\tilde D,0,y_2\tilde D)$, and that its image has dimension $2i$ if $\rk_p\{y_1,y_2\}=2$ or $i$ if $\rk_p\{y_1,y_2\}=1$.
\end{proof}

\subsection{Finishing the proof of Proposition \ref{prop:error}}

The sum \eqref{eq:error} can be broken up into parts (a)-(e) in accordance with Proposition \ref{prop:eqdist_err}. The sum corresponding to (a) can be seen to vanish by comparing with \eqref{eq:rk2}: in place of the $p^{-(2n-1)}$ factor at the front of \eqref{eq:rk2}, we would have $p^{-n(n+1)/2+\dim\ker T_{x,y}}$ from (a) --- notice that in case (a), $\dim\ker T_{x,y}$ is independent of the choice of $x$'s or $y$'s --- which vanishes faster. Case (b) follows by the same argument; in this case, we are summing over an even smaller set of vectors. 

In case (c), we would be considering the sum
\begin{align*}
\frac{1}{p^{-e_c}} \sum_{x_1\in\Z^n\backslash\{0\}}
\sum_{\substack{y_1\in\Z^{n-i} \\ y_1',y_2'\in\Z^i}} 
\sum_{\substack{0 \leq k_x\neq k_y < p\\ (x_1;y_1,y_1')J(k_xx_1;k_yy_1,y_2')^\intercal \equiv_p=0}}f(p^{-\frac{1}{2}}(x_1;y_1,y_1'),p^{-\frac{1}{2}}(k_xx_1;k_yy_1,y_2')),
\end{align*}
for some $e_c \geq 2n-1$. For each $x_1\neq 0$, there exist at most $O(\sqrt p)$ values of $k_x$ that contribute. Similarly, for each $y_1 \neq 0$, there exist at most $O(\sqrt p)$ contributing values of $k_y$; if $y_1=0$ we can do a separate (easier) analysis which we omit. The two outer sums contribute additional $O(p^{(2n+i-1)/2})$. Thus the total sum has size at most
\begin{align*}
O(p^{-(2n-1)+1+n+(i-1)/2})=O(p^{-n+i/2+3/2}),
\end{align*}
which vanishes as $p\rightarrow\infty$ except when $n=2$ and $i=1$. This case requires slightly more attention, and will be handled separately below.

Case (d) can be easily compared with \eqref{eq:rk1} to see that it is $o_p(1)$. Case (e) can be compared with \eqref{eq:rk2} if $\rk_p\{y_1,y_2\}=2$, and \eqref{eq:rk1} if $\rk_p\{y_1,y_2\}=1$; in both cases, the sum is possibly divided by the same power of $p$, but the sum itself is taken over a much smaller lattice, causing it to vanish.

\subsection{Case (c), $\boldsymbol{n=2}$ and $\boldsymbol{i=1}$}

In this case, the above sum is no greater than
\begin{align*}
\frac{1}{p^3}\underset{(a,b;c,d)J(k_xa,k_xb;k_yc,d')^\intercal \equiv_p0}{\sum_{a,b,c,d,d'\in\Z}\sum_{0\leq k_x,k_y <p}} f(p^{-\frac{1}{2}}(a,b;c,d),p^{-\frac{1}{2}}(k_xa,k_xb;k_yc,d')).
\end{align*}
We will obtain the necessary saving by carefully counting the number of the triples $(a,b,k_x)$ such that $a,b,k_xa,k_xb$ (mod $p$) are all bounded by $Cp^{1/2}$ for a fixed constant $C$; here, to be more precise, by an integer mod $p$, we mean its residue class representative in $(-p/2,p/2)$. We will show that they take up $O(p\log p)$, which is enough for the above sum to vanish.

If $a$ or $b$ were zero, there are $O(p)$ such triples. Henceforth we assume $a,b$ are both nonzero. Let $x,y,z,w$ be integers such that $x\equiv a,y\equiv b,z\equiv k_xa,w\equiv k_xb$ (mod $p$). These satisfy the equation
\begin{align} \label{eq:xw-yz}
xw-yz \equiv 0 \ (\mbox{mod $p$}).
\end{align}
Conversely, if $x\neq 0,y\neq 0,z,w$ satisfy \eqref{eq:xw-yz}, we can recover the values of $a,b,k_x$. Thus it suffices to show that the number of (integer) solutions $(x,y,z,w) \in [-Cp^{1/2},Cp^{1/2}]^4,$ $x,y\neq 0$ to \eqref{eq:xw-yz} is $O(p\log p)$. 

Any solution to \eqref{eq:xw-yz} must satisfy $xw-yz=cp$ for some $|c|\leq 2C^2$. It suffice to show that for each such $c$ the number of solutions in $[-Cp^{1/2},Cp^{1/2}]^4$, with $x,y\neq 0$, is $O(p\log p)$. Suppose $\gcd(x,y)=g$, and write $x=g \xi$ and $y=g \zeta$. The set of pairs $(z,w)$ satisfying $xw-yz=cp$ form points on a line segment of length $O(p^{1/2})$, separated by a distance of $\smash{\sqrt{\xi^2+\zeta^2}}$. Therefore, there are $\ll p^{1/2}/\sqrt{\xi^2+\zeta^2}$ solutions for each pair $(x,y)$, and in total, there are
\begin{align*}
\sum_{g=1}^{Cp^{1/2}}\sum_{\substack{\gcd(\xi,\zeta) = 1}} {\bf 1} \!\! \left[ \sqrt{\xi^2+\zeta^2}=O(p^{\frac{1}{2}} g^{-1}) \right] O\!  \left(\frac{p^\frac{1}{2}}{\sqrt{\xi^2+\zeta^2}}\right)
\end{align*}
solutions. By summation by parts, one can show this is at most
\begin{align*}
\sum_{g=1}^{Cp^{1/2}} O \! \left(\frac{p}{g}\right) = O \! \left(p\log p\right)
\end{align*}
as desired.

\section{Application I: mean square discrepancy in lattice point counting} \label{sec:app_disc}

\subsection{Setup} \label{subsect:application_discrepancy}

Let $B\subseteq \R^{2n}$ be the ball at the origin of radius $R$. As in \cite{Sch60} and \cite{KY17}, for a lattice $\Lambda\subseteq\R^{2n}$ of unit covolume, define the \textit{lattice-point discrepancy}
\begin{align*}
D(\Lambda,B) = \left| \frac{|\Lambda \cap B|}{\mathrm{vol}(B)} - 1 \right|.
\end{align*}
We are then interested in understanding the mean square discrepancy
\begin{align*}
\int_{\ylat_{2n}} |D(\Z^{2n}g,B)|^2d\mu_n 
&= \frac{1}{\mathrm{vol}(B)^2}\int_{\ylat_{2n}} \Biggl( \, \sum_{\substack{z\in\Z^{2n}\\ z\neq 0}} \mathbf{1}_B(zg) - \mathrm{vol}(B)\Biggr)^{\! \! 2} d\mu \\
&= \frac{1}{\mathrm{vol}(B)^2}\int_{\ylat_{2n}} \sum_{\substack{z_1,z_2\in\Z^{2n}\\ z_1,z_2\neq 0}} \mathbf{1}_B(z_1g)\mathbf{1}_B(z_2g)-\mathrm{vol}(B)^2 d\mu.
\end{align*}
The contribution from the linearly dependent pairs $z_1,z_2$ (\emph{cf.} Section \ref{sec:rk1_dep}), which coincides with the corresponding part in the original Rogers formula (Theorem \ref{thm:RogersIntegrationFormula}) and is thus already well-understood, is given by
\begin{align} 
\frac{1}{\zeta(2n)}\sum_{\substack{j\geq1\\ k \neq 0}} \int_{\R^{2n}}\mathbf{1}_B(jz)\mathbf{1}_B(kz)dz &= \frac{2\mathrm{vol}(B)}{\zeta(2n)}\sum_{j,k\geq 1}\frac{1}{\max(j^{2n},k^{2n})} \notag \\
&= \frac{2\mathrm{vol}(B)}{\zeta(2n)}\sum_{r\geq1}\frac{2r-1}{r^{2n}} \notag \\
&= \mathrm{vol}(B)\cdot\left(\frac{4\zeta(2n-1)}{\zeta(2n)}-2\right).\label{eq:rk1_pointcounting}
\end{align}
Therefore, the main concern is the contribution from the linearly independent pairs $z_1,z_2$. Recall our rank-2 part of the second moment for the Siegel transform of $\mathbf 1_B$:
\begin{align*}
\frac{1}{\zeta(2n)}
\sum_{\substack{j\geq1 \\ k \in \Z}}
\int_{\R^{2n}}\mathbf 1_B(jz)\mathrm{vol}(L(k,z)\cap B)\frac{dz}{\|z\|}.
\end{align*}
We want to consider the difference between this expression and $\mathrm{vol}(B)^2$, namely
\begin{align} \label{eq:discrepancy_total}
\frac{1}{\zeta(2n)}\sum_{j\geq1}\int_{\R^{2n}}\mathbf 1_B(jz)\left(\sum_{k\in\Z}\frac{1}{\|z\|}\mathrm{vol}(L(k,z)\cap B)-\mathrm{vol}(B)\right)dz.
\end{align}
This leads us to consider the integral
\begin{align} \label{eq:discrepancy_individual}
\int_{\R^{2n}}\mathbf 1_B(jz)\left(\sum_{k\in\Z}\frac{1}{\|z\|}\mathrm{vol}(L(k,z)\cap B)-\mathrm{vol}(B)\right)dz
\end{align}
for each $j \geq 1$. Writing $V_m$ and $S_m$ for the volume and surface area of the $m$-dimensional unit ball, respectively, and writing in polar coordinates, the above is equal to
\begin{align*}
&S_{2n}\int_0^{R/j}\left(\sum_{k\in\Z}\frac{V_{2n-1}}{r}\left(R^2-\frac{k^2}{r^2}\right)_+^{n-\frac{1}{2}}-V_{2m}R^{2n}\right)r^{2n-1}dr \\
&= S_{2n}\int_0^{R/j}\left(\sum_{k\in\Z}\frac{V_{2n-1}}{r}\left(R^2r^2-k^2\right)_+^{n-\frac{1}{2}}-V_{2m}R^{2n}r^{2n-1}\right)dr,
\end{align*}
where the $+$ subscript means the expression is to be interpreted as zero if the inside of the parenthesis is negative. Making the substitution $w=Rr$, this becomes $S_{2n}$ times
\begin{align*}
&\int_0^{R^2/j} \frac{1}{w}\sum_{k\in\Z}V_{2n-1}(w^2-k^2)_+^{n-\frac{1}{2}}-V_{2n}w^{2n-1}dw 
= \int_0^{R^2/j}\Delta(w)\frac{dw}{w},
\end{align*}
where
\begin{align*}
\Delta(w):=\sum_{k\in\Z}V_{2n-1}(w^2-k^2)_+^{n-\frac{1}{2}}-V_{2n}w^{2n}.
\end{align*}

\begin{proposition}\label{prop:bessel}
We have
\begin{align*}
\Delta(w) = 2\sum_{m=1}^\infty\left(\frac{w}{m}\right)^nJ_n(2\pi mw),
\end{align*}
where $J_\nu$ is the Bessel function of the first kind.
\end{proposition}
\begin{proof}
Let $f_w(x) = V_{2n-1}(w^2-x^2)_+^{n-1/2}$. By the Poisson summation formula,
\begin{align*}
\sum_{k\in\Z}f_w(k)=\sum_{m\in\Z}\widehat{f}_w(m),
\end{align*}
so the desired statement will follow upon the computation of the Fourier coefficients $\widehat{f}_w(m)$ --- it will also turn out that $\widehat{f}_w$ decays quickly enough for the legitimate application of the Poisson summation formula. First notice
\begin{align*}
\widehat{f}_w(0)=\int V_{2n-1}(w^2-x^2)_+^{n-\frac{1}{2}}dx=V_{2n}w^{2n}.
\end{align*}

In general, since $f_w$ is even,
\begin{align*}
\widehat{f}_w(m) &= \int V_{2n-1}(w^2-x^2)_+^{n-\frac{1}{2}}e^{-2\pi i m x}dx\\
&=\int_{-w}^w V_{2n-1}(w^2-x^2)^{n-\frac{1}{2}}\cos(2\pi m x)dx,
\end{align*}
and $\widehat f_w(m)= \widehat f_w(-m)$. Hence we assume $m \geq 1$ for the rest of the proof.

Substituting $x=wt$, we have
\begin{align*}
\widehat f_w(m)=V_{2n-1}w^{2n}\int_{-1}^1 (1-t^2)^{n-\frac{1}{2}}\cos(2\pi mwt)dt.
\end{align*}
Recalling the Poisson integral for the Bessel function
\begin{align*}
\int_{-1}^1(1-t^2)^{\nu-\frac{1}{2}}\cos(zt)dt = \pi^\frac{1}{2}\Gamma\left(\nu+\frac{1}{2}\right)\left(\frac{2}{z}\right)^\nu J_\nu(z),
\end{align*}
the formula $V_{2n-1} = \pi^{n-1/2}/\Gamma(n+1/2)$, 
and taking $\nu=n, z=2\pi m w$,  we obtain
\begin{align*}
\widehat f_w(m)=\left(\frac{w}{m}\right)^nJ_n(2\pi mw).
\end{align*}
This completes the proof.
\end{proof}

Thanks to the above proposition, \eqref{eq:discrepancy_individual} becomes
\begin{align*}
2S_{2n}\int_0^{R^2/j} \sum_{m=1}^\infty\left(\frac{w}{m}\right)^nJ_n(2\pi mw)\frac{dw}{w}.
\end{align*}

\subsection{Estimate for small $\boldsymbol j$}
To proceed, we recall Hankel's expansion \cite[10.17.3]{DLMF}
\begin{align}\label{eq:hankel}
J_\nu(z)=\sqrt\frac{2}{\pi z}\left(\cos (z-\theta_n)+\frac{a_1}{z}\sin(z-\theta_n) + O\left(\frac{1}{z^2}\right)\right)
\end{align}
for $z \gg 0$, where $\theta_n=(2n+1)\pi/4$, and $a_1$ is an explicit constant depending on $\nu$. Choose a constant $R_0>1$ such that the above expansion holds for $\nu=n, z\geq2\pi w$, and $w \geq R_0^2$.

We first assume $R^2/j > R_0^2$, and estimate the integral
\begin{align*}
\int_0^{R^2/j}\Delta(w)\frac{dw}{w};
\end{align*}
the cases $R^2/j < R_0^2$ will be dealt with later. First notice that
\begin{align*}
\int_0^{R_0^2} \sum_{m=1}^\infty\left(\frac{w}{m}\right)^nJ_n(2\pi mw)\frac{dw}{w} = O_n(1),
\end{align*}
since the left-hand side is independent of $R$. From $R_0^2$ to $R^2/j$, we plug in \eqref{eq:hankel} and integrate term by term. The cosine term contributes, by integration by parts,
\begin{align*}
&\int_{R_0^2}^{R^2/j} \sum_{m=1}^\infty \frac{w^{n-\frac{3}{2}}}{\pi m^{n+\frac{1}{2}}}\cos(2\pi mw-\theta_n)dw\\
&=\frac{(R^2/j)^{n-\frac{3}{2}}}{2\pi^2}\sum_{m=1}^\infty\frac{1}{m^{n+\frac{3}{2}}}\sin(2\pi m R^2/j -\theta_n)+O_n((R^2/j)^{n-\frac{5}{2}}).
\end{align*}
By similar computations, the remaining terms from \eqref{eq:hankel} can be seen to contribute $O_n((R^2/j)^{n-5/2})$. Hence, for each $j$ such that $j<R^2/R_0^2$, \eqref{eq:discrepancy_individual} is equal to
\begin{align}\label{eq:disc_ind_small_j}
\frac{R^{2n-3}}{2\pi^2j^{n-\frac{3}{2}}}\sum_{m=1}^\infty\frac{1}{m^{n+\frac{3}{2}}}\sin(2\pi m R^2/j -\theta_n) + O_n\left(\frac{R^{2n-5}}{j^{n-\frac{5}{2}}}\right).
\end{align}

\subsection{Estimate for large $\boldsymbol j$}
We now (crudely) estimate the sum of the terms \eqref{eq:discrepancy_individual} over all $j$ such that $R^2/j \leq R^2_0$.

First suppose $1 < R^2/j \leq R_0^2$. Then
\begin{align*}
\int_0^{R^2/j}\Delta(w)\frac{dw}{w} 
&=\int_0^1+\int_1^{R^2/j}\Delta(w)\frac{dw}{w} \\
&=O_n(1) + O_n\left(\int_1^{R^2/j} V_{2n-1}w^{2n-2}+V_{2n}w^{2n-1}dw\right) \\
&=O_n(1)+O_n\left(\frac{R^{4n}}{j^{2n}}\right)=O_n(1).
\end{align*}

In case $R^2/j\leq 1$, we have
\begin{align*}
\int_0^{R^2/j}\Delta(w)\frac{dw}{w} 
&=\int_1^{R^2/j} V_{2n-1}w^{2n-2}-V_{2n}w^{2n-1}dw
=O_n\left(\frac{R^{4n-2}}{j^{2n-1}}\right).
\end{align*}

Therefore, the sum of \eqref{eq:discrepancy_individual} over all $j \geq R^2/R_0^2$ is
\begin{align*}
O_n(1)\cdot O_n(R^2) + O_n\left(R^{4n-2}\int_{R^2}^\infty\frac{1}{j^{2n-1}}dj\right) 
= O_n(R^2).
\end{align*}

\subsection{Conclusion}

Temporarily let us write for convenience
\begin{align*}
M(j,R) =\frac{1}{2\pi^2j^{n-\frac{3}{2}}}\sum_{m=1}^\infty\frac{1}{m^{n+\frac{3}{2}}}\sin(2\pi m R^2/j -\theta_n).
\end{align*}
Collecting the the results from the previous sections, we conclude that \eqref{eq:discrepancy_total} is equal to
\begin{align}\label{eq:disc_aux1}
\frac{S_{2n}}{\zeta(2n)}\cdot R^{2n-3}\sum_{1 \leq j <R^2/R_0^2}M(j,R) + O_n(R^{2n-5}+R^2).
\end{align}

Since this is of order $O_n(R^{2n-1})$, and thus $o(\mathrm{vol}(B))$, it is already meaningful for all $n \geq 2$, . But if $n \geq 4$, we can say more. In such cases, it holds that
\begin{align}\label{eq:disc_aux2}
\left(\zeta\biggl( \! n-\frac{3}{2}\biggr)-1\right)\zeta\biggl(\! n+\frac{3}{2}\biggr) 
< 2-\zeta\biggl(\! n+\frac{3}{2}\biggr),
\end{align}
or equivalently
\begin{align*}
\zeta\biggl(\! n-\frac{3}{2}\biggr)\zeta\biggl(\! n+\frac{3}{2}\biggr) < 2.
\end{align*}

The left-hand side of \eqref{eq:disc_aux2} bounds from above the contributions of the terms $2 \leq j < R^2/R_0^2$ to the sum in \eqref{eq:disc_aux1}. Likewise, the right-hand side of \eqref{eq:disc_aux2} bounds from below the maximal possible contribution from the $j=1$ term. Thus, \eqref{eq:disc_aux2} implies that there exist infinitely many values of $R$ such that 
$$\sum_{1\leq j <R^2/R_0^2} M(j,R) > 0,$$
and also infinitely many values of $R$ such that
$$\sum_{1\leq j <R^2/R_0^2} M(j,R) < 0.$$

Consequently, for $n \geq 4$, \eqref{eq:discrepancy_total} oscillates at an amplitude of $O_n(R^{2n-3})$. For large values of $n$ and $R$, it is well-approximated by
\begin{align*}
\frac{S_{2n}}{2\pi^2\zeta(2n)}\cdot R^{2n-3}\sin(2\pi R^2-\theta_n),
\end{align*}
which is the contribution from the $j=1$ term.

In summary, we obtain the following asymptotic on the mean square discrepancy.
\begin{theorem}\label{thm:mean_sq_disc}
For $n \geq 2$ and the ball $B\subseteq\R^{2n}$ at origin of radius $R$,
\begin{align*}
\int_{\ylat_{2n}} |D(\Z^{2n}g,B)|^2d\mu_n = \frac{1}{\mathrm{vol}(B)}\left(\frac{4\zeta(2n-1)}{\zeta(2n)}-2+O_n(R^{-3})\right).
\end{align*}
Moreover, if $n \geq 4$, then the error changes its sign infinitely many times as $R \rightarrow \infty$. More precisely, the error can be written as 
\begin{align*}
\frac{1+O(2^{-n})}{\pi^2\zeta(2n)}\cdot R^{-3}\sin(2\pi R^2-\theta_n) + O_n(R^{-5})
\end{align*}
where $\theta_n = (2n+1)\pi/4$.
\end{theorem}

\section{Application II: on the result of Holm \cite{Hol23}} 

\subsection{Setup}
In Holm \cite{Hol23} the author is interested in computing the second moment (for $N \gg 0$) of the expression ((5.1) of \cite{Hol23})
\begin{align*}
\frac{1}{\sqrt N}\sum_{m=0}^{N-1}\left(\sum_{z\in\Z^{2n}\backslash\{0\}}(\chi_2 \circ b^m)(zg) - V\right)
\end{align*}
where 
\begin{align*}
    \Omega_2=\{(x;y)\in\R^{2n}:1 \leq \|x\|\|y\| \leq 2, \, 1\leq \|y\|<2\}
\end{align*}
with volume $V = S_n^2(2^n-1)(\log 2) /n$, where $S_n$ is the surface area of the unit ball in $\R^n$. Also let $\chi_2$ be the characteristic function of $\Omega_2$, and $b$ be the matrix $\mathrm{diag}(2,\ldots, 2, \vcenter{\hbox{\scalebox{0.83}{$\tfrac{1}{2}$}}}, \ldots, \vcenter{\hbox{\scalebox{0.83}{$\tfrac{1}{2}$}}}) \in \Sp(2n, \R).$
In these notations, the task is to understand the integral of
\begin{align} \label{eq:holm_goal}
\frac{1}{N}\sum_{0 \leq m,\ell < N}\left( \sum_{z_1,z_2\in\Z^{2n}\backslash\{0\}}\chi_2(z_1gb^m)\chi_2(z_2gb^\ell) - V^2\right)
\end{align}
over $\ylat_{2n}=\Sp(2n,\Z)\backslash\Sp(2n,\R)$.

As in the case of $\xlat_n$ \cite{BG19}, the expectation is that the main contribution comes from the linearly dependent pairs of vectors
\begin{align}\label{eq:holm_dep}
\frac{1}{N}\sum_{0 \leq m,\ell < N}
\sum_{\substack{z_1,z_2 \in \Z^{2n}\backslash\{0\} \\ x_1,x_2\, \mathrm{dep.}}} \chi_2(z_1gb^m)\chi_2(z_2gb^\ell).
\end{align}

Indeed, we will show below that the independent pairs
\begin{align}\label{eq:holm_indep}
\frac{1}{N}\sum_{0 \leq m,\ell < N}
\sum_{\substack{z_1,z_2 \in \Z^{2n}\backslash\{0\} \\ x_1,x_2\, \mathrm{indep.}}} (\chi_2(z_1gb^m)\chi_2(z_2gb^\ell) - V^2).
\end{align}
contribute $O\left(\exp(-(\mathrm{const})\cdot n\log n)\right)$ multiple of $V$, resolving the difficulty from \cite{Hol23}. As a result, we obtain

\begin{theorem} \label{thm:holm_improvement}
The average of \eqref{eq:holm_goal} over $\ylat_{2n}$, as $N\rightarrow\infty$, converges to
\begin{align*}
2V+(\mathrm{error}(n)),
\end{align*}
where the absolute value of the error is no greater than
\begin{align*}
30V\frac{\zeta(\frac{3n}{2})\zeta(\frac{n}{2})}{\zeta(2n)}\frac{(2\pi)^\frac{n}{2}}{\Gamma(\frac{n}{2}+1)} + \frac{V}{n2^{2n}}.
\end{align*}

As a consequence, the variance $\sigma^2$ in \cite[Theorem 1.1]{Hol23} is $2+o_n(1)$.
\end{theorem}

For $n=40$, the above error is already quite small relative to $V$, by a factor of $\approx0.12$.

\subsection{Independent pairs}

To discuss \eqref{eq:holm_indep}, we introduce a few more notations. For an integer $m\geq 1$, let
\begin{align*}
\Omega_{2^m}&=\bigcup_{i=0}^{m-1}\Omega_2b^{-i} 
=\Big\{ (x;y)\in\R^{2n}:1 \leq \|x\|\|y\| \leq 2, 1\leq \|y\|<2^m\Big\}.
\end{align*}
Accordingly, write $\chi_{2^m}$ for the characteristic function of $\Omega_{2^m}$. Then, by the symmetry in the expression \eqref{eq:holm_indep}, and the right-invariance of $\mu$, the integral of \eqref{eq:holm_indep} over $X$ can be rewritten as
\begin{align*}
&\frac{2}{N}\sum_{0\leq m < N}\left(\frac{1}{\zeta(2n)}
\sum_{\substack{j\geq 1 \\ k\in\Z}}
\int_{\R^{2n}}\chi_{2^{m+1}}(jz_1)\int_{L(k,z_1)}\chi_2(z_2)\frac{dz_2}{\|z_1\|}dz_1 - (m+1)V^2\right) \\
&-\frac{1}{N}\sum_{0\leq m < N}\left(\frac{1}{\zeta(2n)}
\sum_{\substack{j\geq 1 \\ k\in\Z}}
\int_{\R^{2n}}\chi_{2}(jz_1)\int_{L(k,z_1)}\chi_2(z_2)\frac{dz_2}{\|z_1\|}dz_1 - V^2\right).
\end{align*}

We will treat each expression of the form
\begin{align*}
\int_{\R^{2n}}\chi_{2^m}(jz_1)\sum_{k\in\Z}\int_{L(k,z_1)}\chi_2(z_2)\frac{dz_2}{\|z_1\|}dz_1-j^{-2n}mV^2
\end{align*}
for fixed $j,m>0$ one at a time.

By the combination of the Poisson summation formula and the Fourier slice theorem, as in Section \ref{subsect:relation_between_formulas} above, this becomes
\begin{align*}
\int_{\R^{2n}}\chi_{2^m}(jz_1)\sum_{l \in \Z}\widehat \chi_2(lz_1J)dz_1-j^{-2n}mV^2.
\end{align*}
The $l=0$ term yields precisely $j^{-2n}mV^2$. Hence we are left to considering
\begin{align*}
\int_{\R^{2n}}\chi_{2^m}(jz_1)\sum_{l \neq 0}\widehat \chi_2(lz_1J)dz_1.
\end{align*}
By writing the Fourier transform $\widehat \chi_2$ as an integral, and changing the order of the integration, this becomes
\begin{align}\label{eq:chi_diff}
j^{-2n}\sum_{l\neq 0} \int_{\Omega_2}\widehat\chi_{2^m}\bigl(\tfrac{l}{j}z_2J \bigr)dz_2.
\end{align}

At this point, we investigate each $l$'th term. Write $z_2=(x_2;y_2)$, and $\xi=(l/j)z_2J=((l/j)y_2,-(l/j)x_2)=(\xi_x;\xi_y)$. We have
\begin{align}\label{eq:chi_hat}
\widehat\chi_{2^m}(\xi)=\int_{1\leq\|y\|<2^m}e^{-2\pi i\langle\xi_y,y\rangle}\int_{\frac{1}{\|y\|}\leq\|x\|\leq\frac{2}{\|y\|}}e^{-2\pi i\langle\xi_x,x\rangle}dxdy.
\end{align}

By the well-known formula
\begin{align*}
\int_{\|x\|\leq R}e^{-2\pi i\langle\xi,x\rangle}dx =\left(\frac{R}{\|\xi\|}\right)^\frac{n}{2}J_{\frac{n}{2}}(2\pi R\|\xi\|),
\end{align*}
we have
\begin{align}
\int_{\frac{1}{\|y\|}\leq\|x\|\leq\frac{2}{\|y\|}}e^{-2\pi i\langle\xi_x,x\rangle}dx &=\left(\frac{2}{\|\xi_x\|\|y\|}\right)^\frac{n}{2}J_\frac{n}{2}\left(\frac{4\pi\|\xi_x\|}{\|y\|}\right) - \left(\frac{1}{\|\xi_x\|\|y\|}\right)^\frac{n}{2}J_\frac{n}{2}\left(\frac{2\pi\|\xi_x\|}{\|y\|}\right) \notag \\
&= \left(\frac{1}{\|\xi_x\|\|y\|}\right)^\frac{n}{2}G\left(\frac{2\pi\|\xi_x\|}{\|y\|}\right), \label{eq:chi_hat_int}
\end{align}
where we wrote
\begin{align*}
G\left(u\right)=2^\frac{n}{2}J_\frac{n}{2}(2u)-J_\frac{n}{2}(u)
\end{align*}
to simplify.

Citing another well-known formula on the Bessel function
\begin{align*}
\int_{\|u\|=1}e^{-2\pi i\rho\langle\xi,u\rangle}du=2\pi\left(\frac{1}{\rho\|\xi\|}\right)^{\frac{n}{2}-1}J_{\frac{n}{2}-1}(2\pi\rho\|\xi\|),
\end{align*}
Plugging this and \eqref{eq:chi_hat_int} into \eqref{eq:chi_hat}, and rewriting $dy$ in the polar coordinates, we have
\begin{align*}
\widehat\chi_{2^m}(\xi)=\frac{2\pi}{\|\xi_x\|^\frac{n}{2}\|\xi_y\|^{\frac{n}{2}-1}}\int_1^{2^m}J_{\frac{n}{2}-1}(2\pi\rho\|\xi_y\|)G\left(\frac{2\pi\|\xi_x\|}{\rho}\right)d\rho.
\end{align*}

From this expression of $\widehat\chi_{2^m}$, and recalling $\xi=(l/j)z_2J$, the $l$'th term of \eqref{eq:chi_diff} is seen to be equal to
\begin{align*}
&\int_{\Omega_2}\widehat\chi_{2^m}(\xi)dz_2 \\
&=\int_{r_y=1}^2\int_{r_x=\frac{1}{r_y}}^{\frac{2}{r_y}}\frac{2\pi S_n^2r_y^{\frac{n}{2}-1}r_x^\frac{n}{2}}{|l/j|^{n-1}}\int_1^{2^m}J_{\frac{n}{2}-1}(2\pi\rho(|l|/j)r_x)G\left(\frac{2\pi(|l|/j)r_y}{\rho}\right)d\rho dr_ydr_x.
\end{align*}
It suffices to consider the case $l > 0$, as this expression is independent of the sign of $l$.

Using the formula
\begin{align*}
\frac{d}{dz}(z^\nu J_{\nu}(z))=z^\nu J_{\nu-1}(z),
\end{align*}
we have
\begin{align*}
\int_{r_x=\frac{1}{r_y}}^{\frac{2}{r_y}}r_x^\frac{n}{2}J_{\frac{n}{2}-1}(2\pi\rho(l/j)r_x)dr_x=\frac{1}{2\pi\rho(l/j)r_y^\frac{n}{2}}G\left(\frac{2\pi(l/j)\rho}{r_y}\right).
\end{align*}
Hence
\begin{align*}
\int_{\Omega_2}\widehat\chi_{2^m}(\xi)dz_2&=S_n^2(l/j)^{-n}\int_{1}^2\frac{dr_y}{r_y}\int_{1}^{2^m}\frac{d\rho}{\rho}G\left(\frac{2\pi(l/j)r_y}{\rho}\right)G\left(\frac{2\pi(l/j)\rho}{r_y}\right) \\
&=S_n^2(l/j)^{-n}\int_{1}^2\frac{dr_y}{r_y}\int_{\frac{1}{r_y}}^{\frac{2^m}{r_y}}\frac{dt}{t}G\left(\frac{2\pi(l/j)}{t}\right)G\left({2\pi(l/j)t}\right),
\end{align*}
by substituting $t=\rho/r_y$.

We are almost at the final stage of our estimate. We bound the above expression by
\begin{align*}
&S_n^2(l/j)^{-n}\int_{1}^2\frac{dr_y}{r_y}\int_{0}^{\infty}\frac{dt}{t}\left|G\left(\frac{2\pi(l/j)}{t}\right)G\left({2\pi(l/j)t}\right)\right| \\
&= S_n^2(l/j)^{-n}\log 2\int_{0}^{\infty}\frac{dt}{t}\left|G\left(\frac{2\pi(l/j)}{t}\right)G\left({2\pi(l/j)t}\right)\right| \\
&= S_n^2(l/j)^{-n}2\log 2\int_{1}^{\infty}\frac{dt}{t}\left|G\left(\frac{2\pi(l/j)}{t}\right)G\left({2\pi(l/j)t}\right)\right|
\end{align*}
(the last equality by the symmetry of the integral under $t\leftrightarrow t^{-1}$). It remains to bound the integral in this expression.

We recite the bounds
\begin{align}\label{eq:besselbound}
|J_\nu(z)|\leq 1, |J_\nu(z)|\leq\frac{(z/2)^\nu}{\Gamma(\nu+1)}
\end{align}
\cite[10.14 and 10.14.4]{DLMF}. We apply the trivial bound to $G(2\pi(l/j)t)$, and the other bound to $G(2\pi l/(jt))$, so that
\begin{align*}
\left|G\left(\frac{2\pi l}{jt}\right)\right|\leq (2^\frac{n}{2}+1)\frac{(2\pi l/(jt))^\frac{n}{2}}{\Gamma(\frac{n}{2}+1)}.
\end{align*}
Plugging this into the above, we obtain
\begin{align*}
\int_{\Omega_2} \widehat\chi_{2^m}\biggl(\frac{l}{j}z_2J \biggr) dz_2
\leq l^{-\frac{n}{2}}j^\frac{n}{2} S_n^2 (2^\frac{n}{2}+1)^2 \frac{4\log 2 \cdot(2\pi)^\frac{n}{2}}{n\Gamma(\frac{n}{2}+1)}.
\end{align*}
Recall $V = S_n^2(2^n-1)(\log 2) /n$. Hence, for a moderately large $n$ (say, $n \geq 10$), the above is at most
\begin{align*}
l^{-\frac{n}{2}}j^\frac{n}{2}5V\frac{(2\pi)^\frac{n}{2}}{\Gamma(\frac{n}{2}+1)}.
\end{align*}

Consequently, \eqref{eq:chi_diff} is bounded by
\begin{align*}
j^{-\frac{3n}{2}}\sum_{l=1}^\infty l^{-\frac{n}{2}}10V\frac{(2\pi)^\frac{n}{2}}{\Gamma(\frac{n}{2}+1)} = j^{-\frac{3}{2}n}\zeta(\frac{n}{2})10V\frac{(2\pi)^\frac{n}{2}}{\Gamma(\frac{n}{2}+1)}.
\end{align*}
It follows that, regardless of $N$, \eqref{eq:holm_indep} is at most
\begin{align*}
30V\frac{\zeta(\frac{3n}{2})\zeta(\frac{n}{2})}{\zeta(2n)}\frac{(2\pi)^\frac{n}{2}}{\Gamma(\frac{n}{2}+1)}.
\end{align*}

\subsection{Dependent pairs}

In the $N$-limit, the integral over $\ylat_{2n}$ of \eqref{eq:holm_dep} converges to
\begin{align}\label{eq:holm_dep_p1}
    \frac{1}{\zeta(2n)} \sum_{s \in \Z} \sum_{\substack{j \geq 1 \\ |k| \geq 1}} 
    \int_{\R^{2n}} \chi_2 (b^{-s} j x ) \chi_2 (k x) \, dx,
\end{align}
provided this expression is finite, which we will show below. 
Note that $\chi_2(j xb^{-s}) \chi_2(k x)$ vanishes unless $j xb^{-s} =: j (x_1;x_2) b^{-s} \in \Omega_2$ and $k x \in \Omega_2,$ which means that
\begin{align} \label{eq:holm_dep_ineq}
j^{-2}, k^{-2} \leq  \Vert x_1 \Vert \, \Vert x_2 \Vert \leq 2j^{-2}, 2k^{-2} 
\hspace{6pt} \text{and} \hspace{8pt}
2^{-s} j^{-1}, k^{-1} \leq \Vert x_2 \Vert \leq 2^{-s+1} j^{-1}, 2 k^{-1} . 
\end{align}
In particular we must have
\begin{align*}
\biggl[ \frac{1}{k^2}, \frac{2}{k^2} \biggr] \cap 
\biggl[ \frac{1}{j^2}, \frac{2}{j^2} \biggr] \neq \varnothing
&\Longleftrightarrow 
\frac{1}{k^2} \leq \frac{2}{j^2} \text{ and } 
\frac{1}{j^2} \leq \frac{2}{k^2} 
\Longleftrightarrow 
\frac{k}{j} \in \biggl[ \frac{1}{\sqrt{2}}, \sqrt{2} \,\biggr];
\end{align*}
and analogously, we must have
\begin{align*}
\biggl[ \frac{1}{2^s j}, \frac{2}{2^s j} \biggr] \cap 
\biggl[ \frac{1}{k}, \frac{2}{k} \biggr] \neq \varnothing
\Longleftrightarrow
\frac{1}{2^s j} \leq \frac{2}{k} \text{ and } \frac{1}{k} \leq \frac{2}{2^s j} 
\Longleftrightarrow \frac{k}{j} \in \bigl[ 2^{s-1} , 2^{s+1} \bigr].
\end{align*}
Combining these bounds, we see that $s \in [-3/2 , 3/2 ] \cap \Z = \lbrace -1,0, 1 \rbrace.$
Therefore, \eqref{eq:holm_dep_p1} is equal to 
\begin{align}\label{eq:holm_dep_p2}
\frac{1}{\zeta(2n)} 
\sum_{s \in \left\lbrace 0, \pm 1 \right\rbrace}
\sum_{j = 1}^\infty 
\sum_{k \neq 0}
\int_{\mathbb{R}^{2n}}
\chi_2 (b^{-s} j x) \chi_2 (k x) \, d x,
\end{align}
where the above discussion allows us to restrict $k,j$ so that
\begin{align*}
|k|/j \in 
\begin{cases}
\bigl[ 1/\sqrt{2}, \, 1 \bigr] &\text{if } s = -1, \\
\bigl[ 1/\sqrt{2}, \, \sqrt{2} \bigr] &\text{if } s = 0, \\
\bigl[ 1, \, \sqrt{2} \bigr] &\text{if } s = 1.
\end{cases}
\end{align*}

Considering the diagonal contributions where $|k| = j$ gives a lower bound on \eqref{eq:holm_dep_p2}. In case $s = 0$, this produces the contribution
\begin{align*}
\frac{1}{\zeta(2n)} \sum_{j = 1}^\infty \frac{2}{j^{2n}} 
\int_{\mathbb{R}^{2n}}
\chi_2 (x)^2 \, d x
=
2V.
\end{align*}
In cases $s= \pm1$, the diagonal contributions can be seen to vanish from \eqref{eq:holm_dep_ineq}.\par 

The remaining non-diagonal terms ($|k|\neq j$) turn out to vanish as $n\rightarrow\infty$. Indeed, we claim that they contribute no more than $V/n 2^{2n}.$ \par 

We first consider the contribution from $s = -1$ and $s = 1.$ These two contributions are in fact equal, since we can sum over $k \in \mathbb{Z}_+$ instead of $k \in \mathbb{Z} \setminus \left\lbrace 0 \right\rbrace$ (at the cost of multiplying the sum by $2$) and then observe that the case-specific double inequalities in $j$ and $k$ actually imply that the roles of these indices are swapped in the two expressions. Thus, the total contribution from the cases $s = \pm 1$ is
\begin{align*}
\frac{4S_n^2}{\zeta(2n)n} \sum_{j,k \geq 1} 
{\bf 1} \bigl[ j < k \leq \sqrt{2}j \bigr] 
\biggl( \frac{2^n}{k^{2n}} - \frac{1}{j^{2n}} \biggr)
\log \frac{k}{j}.
\end{align*}
Only the terms with $2 \leq j < k$ can contribute to this sum. For such indices, the logarithm is non-negative and bounded by $(\log 2)/2,$ and the total contribution is therefore bounded by
\begin{align} \label{lindep-variance-upperbound}
\frac{2\log 2}{\zeta(2n)} \frac{S_n^2}{n}
\sum_{j \geq 2} \sum_{k = j+1}^{\lfloor \smash{\sqrt{2}} j \rfloor} 
\biggl( \frac{\sqrt{2}}{k} \biggr)^{\! 2n}
&=
\frac{2^{n+1}\log 2}{\zeta(2n)} \frac{S_n^2}{n}
\sum_{j \geq 2} \sum_{k = j+1}^{\lfloor \smash{\sqrt{2}} j \rfloor} 
\frac{1}{k^{2n}} .
\end{align}
(Recall that $S_n$ is the surface area of the unit ball in $\R^n$.) 
The inner sum can be estimated with the integral
\begin{align*}
\int_j^{\sqrt{2}j} 
\frac{1}{x^{2n}} \, dx =
\frac{1}{2n-1} \biggl( 1-\frac{1}{2^{n-1/2}} \biggr) \frac{1}{j^{2n-1}}.
\end{align*}
Then the outer sum is a multiple of $-1 + \zeta(2n-1) \leq 2/4^n$. In this way we find that the whole expression \eqref{lindep-variance-upperbound} is at most
\begin{align*}
\frac{2^{n+1} \log 2}{\zeta(2n)}
\frac{S_n^2}{n}
\frac{1}{2n-1} 
\biggl( 1 - \frac{1}{2^{n-1/2}} \biggr)
\frac{n}{n-1} \frac{1}{2^{2n-1}} 
< 
\frac{V}{n2^{2n+1}}
\end{align*}
for $n\geq 2$. 

In case $s = 0$, the nondiagonal terms can be seen to contribute 
\begin{align*}
\frac{2S_n^2}{\zeta(2n)n} 
\sum_{1 \leq j < k}^\infty 
{\bf 1} \Bigl[ k \le \sqrt{2}j \Bigr]
\biggl( \frac{2^n}{k^{2n}} - \frac{1}{j^{2n}} \biggr)
\log \frac{2j}{k}.
\end{align*}
Since the indices satisfy $\smash{\sqrt{2}} \le 2j/k \le 2,$ all terms are non-negative. Moreover, as the term $j = 1$ doesn't contribute, we can estimate the above expression with
\begin{align*}
\frac{(2 \log 2)S_n^2}{\zeta(2n)n} 
\sum_{j \geq 2}^\infty 
{\bf 1} \Bigl[  j+1 \leq  k \leq \sqrt{2}j \Bigr]
\biggl( \frac{2^n}{k^{2n}} - \frac{1}{j^{2n}} \biggr),
\end{align*}
which is identical to \eqref{lindep-variance-upperbound}, so is also bounded by $V/(n2^{2n+1})$. This proves the claim, and completes the proof of Theorem \ref{thm:holm_improvement}.

\section{Application III: Gaussian heuristic}

The term ``Gaussian heuristic'' is frequently used in works on lattice-based cryptography to refer to the fact that the shortest nonzero vector of most lattices in $\SL(n,\Z)\backslash\SL(n,\R)$ has length $\approx \sqrt{n/(2\pi e)}$ --- see, for instance, \cite[Theorem 6]{LN25} or \cite[Theorem 3]{GSVV25} for a rigorous formulation. In this section, we establish the same principle for symplectic lattices.

To this end, we adopt the framework set up by \cite[Section 4]{GSVV25}. The new input from us is the computation of the second moment
\begin{align}\label{eq:2ndmmt_for_GH}
\int_{\ylat_{2n}} \Biggl|\sum_{\substack{z\in\Z^{2n}\\ z\neq0}} \mathbf{1}_B(zg)\Biggr|^2d\mu_n,
\end{align}
where $B\subseteq\R^{2n}$ is a ball of radius $R$ centered at origin. 
The difference from our earlier work in Section \ref{sec:app_disc} is that, there we intended $R$ to be arbitrarily large, but here we have $R = O(\sqrt{n})$ in mind. This calls for a different kind of analysis.

\begin{proposition}\label{prop:2ndmmt_for_GH}
For $n \geq 3$ and $R>0$ satisfying $2\pi R^2 \geq 1$, \eqref{eq:2ndmmt_for_GH} equals
\begin{align*}
\mathrm{vol}(B)^2 + \mathrm{vol}(B)\left(\frac{2\zeta(2n-1)}{\zeta(2n)} + \mathrm{error}\right),
\end{align*}
where the error is $O(nR^{-\frac{8}{3}} + nR^{-2}2^{-n}).$
\end{proposition}
\begin{proof}
Following Section \ref{subsect:application_discrepancy} and Proposition \ref{prop:bessel}, \eqref{eq:2ndmmt_for_GH} equals
\begin{align*}
\mathrm{vol}(B)^2 + \frac{2\zeta(2n-1)}{\zeta(2n)}\mathrm{vol}(B) + \frac{2S_{2n}}{\zeta(2n)}\sum_{j=1}^\infty\sum_{m=1}^\infty\int_0^{R^2/j} \left(\frac{w}{m}\right)^nJ_n(2\pi mw)\frac{dw}{w}.
\end{align*}
For the integral inside the double sum, we substitute $z=2\pi m w$ to rewrite it as
\begin{align}\label{eq:GH_1}
\frac{1}{(2\pi m^2)^n}\int_0^{y} z^{n-1}J_n(z)dz,
\end{align}
where $y=2\pi mR^2/j$.

We apply the formula \cite[10.22.10]{DLMF} with $\mu=n-1, \nu=n$. After a straightforward computation, \eqref{eq:GH_1} can be seen to be equal to
\begin{align*}
\frac{1}{(2\pi m^2)^n} y^{n-1}\sum_{k=0}^\infty(n+1+2k)\mathrm{B}(n,k+1)J_{n+1+2k}(y).
\end{align*}
where $\mathrm{B}$ is the beta function. By the result of \cite{Lan00} and Lemma \ref{lemma:funnysum} below, the above can be bounded in size by
\begin{align*}
\frac{3}{(2\pi m^2)^n}y^{n-1}\min(y^{-\frac{1}{3}},1).
\end{align*}

Recalling what $y$ was, and by seperating out the contribution from the $j=m=1$ term, provided that $2\pi R^2 \geq 1$ we have
\begin{align*}
\frac{2S_{2n}}{\zeta(2n)}\sum_{j,m=1}^\infty \frac{3}{(2\pi m^2)^n}y^{n-1}\min(y^{-\frac{1}{3}},1) &= O(nV_{2n}R^{2n-\frac{8}{3}} +nV_{2n}R^{2n-2}2^{-n}) \\
&=\mathrm{vol}(B)\cdot O(nR^{-\frac{8}{3}} + nR^{-2}2^{-n}).
\end{align*}
This completes the proof.
\end{proof}

\begin{lemma}\label{lemma:funnysum}
For $n \geq 3$
$$\sum_{k=0}^\infty(n+1+2k)\mathrm{B}(n,k+1) = \frac{n}{n-2}.$$
For $n=2$ the sum diverges.
\end{lemma}
\begin{proof}
First assume $n\geq 3$. Plugging in the integral formula for the beta function, we can rewrite the left-hand side as
\begin{align*}
\sum_{k=0}^\infty (n+1+2k)\int_0^1 t^{n-1}(1-t)^kdt.
\end{align*}
Changing the sum and the integral, we have
\begin{align*}
\int_0^1t^{n-1}\sum_{k=0}^\infty(n+1+2k)(1-t)^kdt.
\end{align*}
Recalling the facts
\begin{align*}
\sum_{k=0}^\infty x^k = \frac{1}{1-x} \mbox{ and } \sum_{k=0}^\infty kx^k = \frac{x}{(1-x)^2}
\end{align*}
for $|x|<1$, the above can be computed as
\begin{align*}
\int_0^1 t^{n-1}\left((n+1)t^{-1} +2(1-t)t^{-2}\right)dt 
&= \int_0^1 t^{n-1}\left((n-1)t^{-1} +2t^{-2}\right)dt \\
&= 1+\frac{2}{n-2} =\frac{n}{n-2}.
\end{align*}

In case $n=2$, the summand in question equals
\begin{align*}
\frac{3+2k}{(k+1)(k+2)},
\end{align*}
from which it is clear that the sum diverges.
\end{proof}

\begin{theorem} \label{thm:symplectic_GH}
For $n$ sufficiently large, $\Lambda \in \ylat_{2n}$ in a subset of measure at least $1-o_n(1)$ satisfies
\begin{align*}
1-\frac{\log\log n}{2n} \leq \left(\frac{V_{2n}}{2}\right)^{\frac{1}{2n}}\lambda_1(\Lambda) \leq 1+\frac{\log\log n}{2n}.
\end{align*}
\end{theorem}
Note that, by Stirling's approximation,
\begin{align*}
\left(\frac{V_{2n}}{2}\right)^{\frac{1}{2n}} \approx \sqrt{\frac{\pi e}{n}}.
\end{align*}
\begin{proof}
Let $\varepsilon(n) = (\log n)^{-1}$, and set $\mathrm{vol}(B) = 2\varepsilon(n)^{-1}$; note that this implies $\mathrm{error} = O(n^{-1/3})$, as defined in Proposition \ref{prop:2ndmmt_for_GH}.

Now by applying Proposition \ref{prop:2ndmmt_for_GH} to Proposition 27 of \cite{GSVV25}, we obtain the following statement. For $\Lambda\in \ylat_{2n}$ in a subset of measure at least $1-o_n(1)$, we have
\begin{align*}
2\varepsilon(n) \leq V_{2n}\lambda_1(\Lambda)^{2n} \leq 2\varepsilon(n)^{-1}.
\end{align*}
By taking $2n$'th roots on all sides, and approximating $\varepsilon(n)^{\pm 1/2n}=\exp(\pm(\log\log n)/2n)$ by a Taylor expansion, the theorem follows.
\end{proof}

\section{Application IV: The lattice smoothing parameter} \label{sec:app_smoothing parameter}
For a lattice $L \subseteq \R^{n}$, the \textit{discrete Gaussian} on $L$ with variance $s^2/2 \pi$ is defined in terms of $\varrho_s(x) := \exp \left( - \pi s^{-2} \|x\|^2 \right)$ as the probability measure
\begin{align*}
    D_s(A) := c(L,s) \sum_{\substack{v \in L}} {\bf 1}_A (v) \varrho_s(v) \qquad (\text{for } A \subseteq L)
\end{align*}
where the constant $c(L,s) = 1/\varrho_s (L)$ normalizes the full mass of $L$ to $1.$ (In this section, we will use the same symbol to denote a function on $\R^n$ and its Siegel transform; this will cause no confusion.) \par
For $\varepsilon > 0,$ the \textit{lattice smoothing parameter}, first introduced in \cite{MR07}, is defined as 
\begin{align*}
    \eta_\varepsilon (L) := 
    \inf \! \Big\{ s > 0 \, \Big| \, \varrho_{1/s}(L^P) \leq 1 + \varepsilon \Big\}
    = 
    \inf \! \Big\{ s > 0 \, \Big| \, \varrho_{1/s}(L^P \setminus \lbrace 0 \rbrace) \leq \varepsilon \Big\}.
\end{align*}
Recent work of Pouly--Shen \cite{PS26} uses the second-moment case of Rogers' integration formula to give metrical bounds on $\eta_\varepsilon$ over the space $\xlat_n$ of unimodular lattices $L \subseteq \R^n.$ Since the NTRU lattices \cite{HPS98}, an important family of constructions in lattice-based cryptography, are symplectic \cite{GHN06}, it is natural to ask about analogous bounds on $\eta_\varepsilon$ over the space $\ylat_{2n}$ as well, and to compare this to the case of $\xlat_n.$ In this section we use Thm. \ref{thm:main-Schwartz} to give an answer to this question in the form of the following theorem.
\begin{theorem}\label{thm:lattice_smoothing_parameter}
    For any $n \geq 2$ and $s > 0$, the following statements hold:
    \begin{enumerate}
        \item The second moment over $\ylat_{2n}$ of the (punctured) discrete Gaussian $\varrho_s$ equals
            \begin{align*}
                \int_{\ylat_{2n}} \varrho_s (L \setminus \lbrace 0 \rbrace )^2 \, d \mu (L)
                =
                s^{2n} \Bigl( s^{2n} + \zeta(n) L(n, \chi) - \zeta(2n) + \tfrac12 \mathrm{Z}_s(n) \Bigr),
            \end{align*}
            where $\chi$ denotes the non-principal Dirichlet character (mod $4$), and $\mathrm{Z}_s$ denotes the (incomplete) Epstein zeta function
            \begin{align*}
                \mathrm{Z}_s (n) = 
                \sum_{\substack{x,y \in \mathbb{Z} \\ x,y \neq 0}}
                \frac{1}{(x^2 s^{-2} + y^2 s^2)^n}.
            \end{align*}
        \item Let $0 <\varepsilon <1$ be arbitrary, and let $\varepsilon^{-1/(2n)} < s < (\varepsilon/2)^{-1/n}$ be a fixed large number. Then, for any $n \geq 2$ there exists a subset $\Theta_{2n} \subseteq \ylat_{2n}$ of measure $\mu(\Theta_{2n}) = 1 - O(2^{-n})$, such that all $L \in \Theta_{2n}$ satisfy $\varrho_{1/s}(L^P \setminus \lbrace 0 \rbrace) \leq \varepsilon$.\\[+0.5em] In other words, as $n \rightarrow \infty$, the smoothing parameter satisfies $\eta_\varepsilon (L) = 1 + O(1/n)$ on a set of measure $1-o(1)$. 
    \end{enumerate}
    
\end{theorem}
\begin{proof}
    Let $\varepsilon > 0$ and write $\Lambda' := \Lambda \setminus \lbrace 0 \rbrace$. To study $\eta_\varepsilon (\Lambda)$, we are required to study the discrete Gaussian $\varrho_{1/s}(\Lambda^P)$ on the dual lattice $\Lambda^P$ of $\Lambda$. However, for our metrical purpose, it is equivalent (and notationally simpler) to study $\varrho_{1/s}(\Lambda)$ instead, as the probability measure $\mu$ on $\ylat_{2n}$ is invariant under the map $\Lambda \mapsto \Lambda^P$.  \\[+0.5em]
    For any $s > 0$, Siegel's mean value theorem over $\ylat_{2n}$ gives $\mathbb{E} [\varrho_s (\Lambda')] = s^{2n}.$ To express the variance of $\varrho_s(\Lambda')$, we consider the function $A(2n) = \zeta(n) L(n, \chi) - \zeta(2n)$. (Note that the function $A(n)$ is the function $K(n)$ from \cite{PS26}.) By using Thm. \ref{thm:main-Schwartz} and a computation analogous to \cite{PS26}, we find that
    \begin{align}\label{discrete_Gaussian_variance}
        \mathrm{Var}(\varrho_s (\Lambda')) = 2 s^{2n} A(2n) + 
        \tfrac12 s^{2n} \mathrm{Z}_s(n),
    \end{align}
    where $\mathrm{Z}_s(n)$ arises from the new symplectic term in the second-moment formula. \par 
    We now replace $s$ by $1/s$ and use Chebyshev's inequality in combination with the trivial identity $\varrho_{s}(\Lambda) = 1 + \varrho_s (\Lambda')$ to obtain 
    \begin{align*}
    \mathbb{P} \Bigl( \varrho_{1/s}(\Lambda) > 1 + \varepsilon \Bigr)
    &=
    \mathbb{P} \Bigl( \varrho_{1/s}(\Lambda^\ast) - 1 -s^{-2n} > \varepsilon -s^{-2n}
    \Bigr)
    \\[0.5em] &\leq 
    \mathbb{P} \Bigl( \Big| \varrho_{1/s}(\Lambda') -s^{-2n} \Big| > 
    \varepsilon -s^{-2n} \Bigr)
    \\[0.5em] &\leq 
    \frac{\mathrm{Var}(\varrho_{1/s}(\Lambda '))}{(\varepsilon - s^{-2n})^2},
    \end{align*}
    where the final inequality holds if $\varepsilon - s^{-2n} > 0$.
    With \eqref{discrete_Gaussian_variance}, we arrive at the bound 
    \begin{align}\label{discrete_Gaussian_prob-bound}
        \mathbb{P} \Bigl( \varrho_{1/s}(\Lambda) > 1 + \varepsilon \Bigr)
        &\leq 
        \frac{2 A(2n) + 
        \sfr12 \mathrm{Z}_{1/s}(n)}
        {s^{2n}(\varepsilon - s^{-2n})^2}.
    \end{align}
    \par Choose $s = s (\varepsilon)$ so that the denominator on the right-hand side of \eqref{discrete_Gaussian_prob-bound} is $2$; explicitly, this means
    \begin{align*}
    s(\varepsilon) = \Biggl( \frac{1 + \varepsilon + \sqrt{2 \varepsilon + 1}}{\varepsilon^2} \Biggr)^{\! \! 1/2n} < (4\varepsilon^{-2})^{1/(2n)} = (\varepsilon/2)^{-1/n} = 1 + O\bigl( \tfrac{1}{n} \bigr).
    \end{align*}
    Then in particular we have $0 < s(\varepsilon)^{-2n} < \varepsilon$, and 
    \begin{align*}
    \mathbb{P} \Bigl( \varrho_{1/s(\varepsilon)}(\Lambda^P) > 1 + \varepsilon \Bigr)
    =
    \mathbb{P} \Bigl( \varrho_{1/s(\varepsilon)}(\Lambda) > 1 + \varepsilon \Bigr)
    &\leq 
    A(2n) + \tfrac14 \mathrm{Z}_{1/s(\varepsilon)}(n).
    \end{align*}
    By definition of the smoothing parameter, we therefore have $\eta_\varepsilon (\Lambda) \leq s(\varepsilon)$ on a set $\Theta_{2n}$ of large measure, viz. 
    \begin{align}
        \mu (\Theta_{2n}) \geq 1 - A(2n) - \tfrac{1}{4} \mathrm{Z}_{1/s(\varepsilon)}(n).
    \end{align}
    It is relatively straightforward to prove the estimates
    \begin{align*}
    \mathrm{Z}_s(n) 
    \sim
    \frac{2}{(s^2+s^{-2})^n} \ll \frac{1}{2^n}
    \quad \text{and} \quad 
    A(2n) 
    \sim 
    \frac{1}{2^n}
    \qquad (\text{as } n \longrightarrow \infty).
    \end{align*}
    (The implied constant in the estimate for $\mathrm{Z}_s(n)$ is independent of $s$.) This proves that, for all $L$ in a large set, $\eta_\varepsilon(L)$ is at most $1 + O(1/n)$. \\[+0.5em]
    Finally, we show that $\eta_\varepsilon$ is also bounded from below: Observe that we have
    \begin{align*}
        \varrho_{1/s}(L^P) = s^{-2n} \varrho_s (L) \geq s^{-2n} \qquad (\text{for any } s > 0)
    \end{align*}
    by Poisson summation and the fact that $\varrho_s(0)$ is always $1$. This shows that
    \begin{align*}
        \Big\{ s > 0 \, \Big| \, \varrho_{1/s}(L^P) \leq 1+ \varepsilon \Big\}
        \subseteq 
        \Big\{ s > 0 \, \Big| \, s \geq (1+\varepsilon)^{-1/2n} \Big\}.
    \end{align*}
    Since $\eta_\varepsilon(L)$ is the infimum of the leftmost set, this inclusion immediately implies that
    \begin{align*}
        \eta_\varepsilon(L) \geq \inf \Big\{ s > 0 \, \Big| \, s \geq (1+\varepsilon)^{-1/2n} \Big\} 
        = 
        (1+\varepsilon)^{-1/2n} 
        = 
        1 + O\bigl( \tfrac{1}{n}\bigr).
    \end{align*}
    The final claim follows, and the proof is therefore concluded.
\end{proof}

\vspace*{10pt}
\noindent $\hrulefill$ \\\par 
{\small 
\textsc{Department of Mathematics, Christian-Albrechts-Universit\"at zu Kiel},\\ \textsc{24118-Kiel, Germany}. \\ \par 
\textsc{Department of Mathematical Sciences, University of Cincinnati},\\ \textsc{Cincinnati, OH, United States}.
}


\begin{thebibliography}{99}

\bibitem[AK16]{AK16} J. Athreya and I. Konstantoulas. Discrepancy of general symplectic lattices. arXiv:1611.07146


\bibitem[And09]{And09}
A. Andrianov. 
Introduction to Siegel Modular Forms and Dirichlet Series, Springer, 2009.

\bibitem[Ban93]{Ban93} W. Banaszczyk. New bounds in some transference theorems in the geometry of numbers. Math. Ann. 296, 625-535 (1993).

\bibitem[BG19]{BG19} M. Bj\"orklund and A. Gorodnik. Central limit theorems for Diophantine approximants. Math. Ann. 374, 1371-1437 (2019).

\bibitem[BG23]{BG23} M. Bj\"orklund and A. Gorodnik. Poisson approximation and Weibull asymptotics in the geometry of numbers. Trans. Amer. Math. Soc. 376 (2023), 2155-2180.


\bibitem[Bor19]{Bor19} A. Borel. Introduction to Arithmetic Groups. University Lecture Series, volume~73. American Mathematical Society, Providence, RI, 2019.

\bibitem[COU01]{COU01} L. Clozel, H. Oh, and E. Ullmo. Hecke operators and equidistribution of Hecke points. Invent. Math. 144 (2001), no. 2, 327-351.

\bibitem[DLMF]{DLMF} National Institute of Standards and Technology. NIST Digital Library of Mathematical Functions. \url{https://dlmf.nist.gov}

\bibitem[Ehr03]{ehrenpreis}
L. Ehrenpreis. \textit{The universality of the Radon transform}. Oxford Mathematical Monographs, Oxford University Press, 2003.

\bibitem[GHN06]{GHN06} N. Gama, N. Howgrave-Graham, and P.Q. Nguyen, P.Q. (Symplectic Lattice Reduction and NTRU. In: Vaudenay, S. (eds) Advances in Cryptology - EUROCRYPT 2006. Lecture Notes in Computer Science, vol 4004. Springer, Berlin, Heidelberg.

\bibitem[GSVV25]{GSVV25} N. Gargava, V. Serban, M. Viazovska and I. Viglino. Module lattices and their shortest vectors. arXiv:2510.12893v1

\bibitem[GSVV25-2]{GSVV25-2} N. Gargava, V. Serban, M. Viazovska and I. Viglino. Integral matrices of fixed rank over number fields. arXiv:2510.11673.


\bibitem[GN12]{GN12} A. Gorodnik and A. Nevo. Counting lattice points. J. Reine Angew. Math. no. 663 (2012), 127-176.

\bibitem[HPS98]{HPS98} J. Hoffstein, J. Pipher, and J. Silverman. NTRU: A ring-based public key cryptosystem. Lecture Notes in Computer Science, 267-288.

\bibitem[HK23]{HK23} T. Horesh and Y. Karasik. Equidistribution of primitive lattices in $\R^n$. Quart. J. Math. Vol. 74, Issue 4, December 2023, Pages 1253-1294.

\bibitem[Hol23]{Hol23} K. Holm. A central limit theorem for counting functions related to symplectic lattices and bounded sets. Discrete Contin. Dyn. Syst. 2023, 43(10): 3667-3705.

\bibitem[KY17]{KY17} D. Kelmer and S. Yu, \textit{The second moment of the Siegel transform in the space of symplectic lattices}, Int. Math. Res. Not., No. 8, 2021, 5825--5859.

\bibitem[Kim23]{Kim23} S. Kim. Mean value formulas on sublattices. J. Number Theory 241 (2022), 330-351.

\bibitem[Kim24]{Kim24} S. Kim. Counting rational points on Grassmannians. Monatsh. Math. (2023) 201:825-864.

\bibitem[Kim24-2]{Kim24-2} S. Kim. Adelic Rogers formula. J. Lond. Math. Soc. 109: e12830.

\bibitem[KS22]{KS22} S. Kim and M. Skenderi. Higher-rank pointwise discrepancy bounds and logarithm laws for generic lattices. Acta Arith. 205 (2022), 227-249.

\bibitem[Lan00]{Lan00} L. Landau. Bessel functions: monotonicity and bounds.  J. Lond. Math. Soc. 61: 197-215.

\bibitem[LN25]{LN25} J. Li and P. Q. Nguyen. A complete analysis of the BKZ lattice reduction algorithm. Journal of Cryptology 38 (2025), no. 1, 12.

\bibitem[MR07]{MR07} D. Micciancio and O. Regev. Worst-case to average-case reductions based on Gaussian measures. SIAM Journal on Computing, 37, 267-302.

\bibitem[PS26]{PS26} A. Pouly and Y. Shen. Solving the Shortest Vector Problem in \(2^{0.63269n+o(n)}\) Time on Random Lattices. In J. Daemen \& E. Thom\'e (Eds.), Advances in Cryptology --- EUROCRYPT 2026 (pp. 92-123). Lecture Notes in Computer Science, Vol. 16544. Springer, Cham.

\bibitem[RS24]{RS24} O. Regev and N. Stephens-Davidowitz. A reverse Minkowski theorem. Ann. of Math. (2) 199(1): 1-49.

\bibitem[Rog47]{Rog47} C.A. Rogers. Existence theorems in the geometry of numbers. Ann. Math. Vol. 48, No. 4, pp. 994-1002.

\bibitem[Rog55]{Rog55} C.A. Rogers. Mean values over the space of lattices. Acta Math. 94, 249-287 (1955).

\bibitem[Sch60]{Sch60} 
W. Schmidt. 
{A metrical theorem in geometry of numbers}. 
Trans. Am. Math. Soc. {\bf 95}, 516--529 (1960).

\bibitem[Sch66]{Sch66} W. Schmidt. Asymptotic formulae for point lattices of bounded determinant and subspaces of
bounded height. Duke Math. J. 35 (1968), 327-339.

\bibitem[Sie45]{Sie45} C.L. Siegel. A mean value theorem in geometry of numbers. Ann. Math. 46(2) (1945), 340-347.

\bibitem[SS19]{SS19} A. Str\"ombergsson and A. S\"odergren. On the generalized circle problem for a random lattice in large
dimension. Adv. Math. 345 (2019), 1042-1074.

\bibitem[Thu93]{Thu93} J. Thunder. Asymptotic estimates for rational points of bounded height on flag varieties. Compositio
Math. 88 (1993), no.2, 155-186.

\bibitem[Wei46]{Wei46} A. Weil. Sur quelques resultats de Siegel, Summa Brasil. Math., 1 (1946),
21-39.

\bibitem[Wid10]{Wid10} M. Widmer. Counting primitive points of bounded height. Trans. Amer. Math. Soc. 362 (2010), no.
9, 4793-4829.
\end{thebibliography}
\end{document}